\documentclass[a4paper,11pt]{article}
\usepackage{graphicx} 

\usepackage[utf8]{inputenc}
\usepackage{amsmath,amssymb,amsthm,mathrsfs}
\usepackage{xcolor}
\usepackage{setspace}
\usepackage[margin=1in]{geometry}
\usepackage{enumitem}
\usepackage{tikz}
\usepackage{pstricks}
\usepackage{pst-all}
\usepackage{verbatim}
\usepackage{caption}
\usepackage{subcaption}
\usepackage[indent=15pt, skip=7pt]{parskip}
\usepackage{hyperref}
\usepackage[normalem]{ulem}

\usepackage{authblk} 

\theoremstyle{plain}

\newtheorem{thm}{Theorem}[section]
\newtheorem{lem}[thm]{Lemma}
\newtheorem{prop}[thm]{Proposition}

\newtheorem{theorem}[thm]{Theorem}

\theoremstyle{definition}
\newtheorem{definition}{Definition}[section]
\newtheorem{remark}[thm]{Remark}
\newtheorem{assumption}{Assumption}[section]

\theoremstyle{definition}
\newtheorem{defi}{Definition}[section]

\newcommand{\N}{\mathbb{N}}
\newcommand{\R}{\mathbb{R}}
\newcommand{\Z}{\mathbb{Z}}

\newcommand{\B}{\mathcal{B}}
\newcommand{\Fcal}{\mathcal{F}}

\newcommand{\e}{\mathrm{e}}
\newcommand{\ds}{\displaystyle}
\newcommand{\dd}{\mathrm{d}}

\nocite{*}

\title{Slowly oscillating periodic solutions in a nonlinear Volterra equation with non-symmetric feedback}
\author[1]{Quentin Griette}
\author[1,2]{Franco Herrera}
\affil[1]{\small LMAH, Université Le Havre Normandie, 76600, Le Havre, France; quentin.griette@univ-lehavre.fr; franco.herrera-granda@etu.univ-lehavre.fr}
\affil[2]{Instituto de Matemáticas, Universidad de Talca, Talca, Chile; franco.herrera@utalca.cl}

\date{\today}

\newcommand{\writefoot}[1]{
    \renewcommand{\thefootnote}{}
    \footnotetext{\hspace{-16.5pt}\scriptsize#1}
    \renewcommand{\thefootnote}{\arabic{footnote}}
}
\numberwithin{equation}{section}

\begin{document}

\maketitle

\writefoot{\small \textbf{AMS subject classifications (2020).} Primary: 34K13; Secondary: 45M15, 45D05, 35B10, 34K25, 34K26. \smallskip}
\writefoot{\small \textbf{Keywords.} Volterra equation,  delay-differential equation, slowly oscillating solution, singular perturbation
\smallskip
}
\writefoot{\small \textbf{Acknowledgements:} 
F.H. acknowledges the support of ANID-Subdirección de Capital Humano/Doctorado Nacional/2024-21240616.
Q.G. acknowledges the support of ANR grant  ``Indyana'' number  ANR-21-CE40-0008. 
F.H. and Q.G. acknowledge the support of the Math AmSud program for project 22-MATH-09 ``STEMDYNEPID''. 
}

\begin{abstract}
    In this work we study a nonlinear Volterra equation with non-symmetric feedback that arises as a particular case of the Gurtin-MacCamy model in population dynamics. We are particularly interested in the existence of  slowly oscillating periodic solutions when the trivial stationary state is unstable.  Here the absence of symmetry of the nonlinearity prevents the use of many traditional strategies to obtain a priori estimates on the solution. Without a precise knowledge of the period of the solution, we manage to prove the forward invariance of a carefully constructed set of initial data  whose properties imply the slowly oscillating character of all continuations. We prove the existence of periodic solutions by constructing a homeomorphism between our set and a convex subset of a different Banach space,  thereby showing that it possesses the fixed-point property.  Finally, in a singular limit of a parameter, we show that this periodic solution converges to the solution of a well-known discrete difference equation. We conclude the paper with some numerical simulations to illustrate the existence of the periodic orbit as well as the singular limit behavior.
\end{abstract}

\section{Introduction}

In this work we are interested in the existence of slowly oscillating solutions for the Volterra-type integral equation
\begin{equation}\label{eq:prob}
    b(t) = \frac{1}{\varepsilon} \int_{1-\varepsilon/2}^{1+\varepsilon/2} f\big(b(t-s)\big)\dd s = \frac{1}{\varepsilon}\int_{t-1-\varepsilon/2}^{t-1+\varepsilon/2} f\big(b(s)\big)\dd s,
\end{equation}
where $\varepsilon\in(0,1)$ is a scaling parameter and $f$ is a negative feedback nonlinearity, i.e. a continuous function satisfying $xf(x)<0$ (see Figure \ref{fig:1} below). 

This equation  arises naturally as a particular case of the  Gurtin-MacCamy equation
\begin{equation}\label{eq:Gurtin-MacCamy}
	\begin{cases}
        (\partial_t + \partial_a)u(t,a) = -\mu u(t,a), & t>0,~ a>0 \\
        u(t,0) = f\left( \int_0^\infty \gamma(a)u(t,a) \dd a \right), & t>0, \quad u(0,\cdot)=u_0\in L^1_+(\R_+), 
\end{cases}
\end{equation}
that describes a theoretical population structured by chronological age. Equation \eqref{eq:Gurtin-MacCamy} has been proposed by Gurtin and MacCamy \cite{Gurtin-MacCamy-1974} and has been extensively studied since then, see e.g. Webb \cite{Webb-1985}, Thieme \cite{Thieme-2003}, Inaba \cite{Inaba-2017}, Magal and Ruan \cite{Magal-Ruan-2018}, and the references therein. Recently,  Ma and Magal \cite{MMa} studied the global asymptotic stability of the unique positive stationary solution of the model \eqref{eq:Gurtin-MacCamy} with a Ricker's-type nonlinearity $f(u)= \alpha ue^{-u}$. They showed that the global asymptotic stability holds whenever $1<\alpha<e^2$ (the stationary solution does not exists for $\alpha<1$) and that a Hopf bifurcation may occur at $\alpha=\alpha^\star>e^2$ depending on the function $\gamma$; and for each $\alpha^\star>e^2$ they provide a $\gamma$ for which a Hopf bifurcation effectively occurs at $\alpha=\alpha^\star$. By reducing \eqref{eq:Gurtin-MacCamy} to a nonlinear convolution equation such as \eqref{eq:prob}, Herrera and Trofimchuk \cite{HT} proved the global asymptotic stability of the unique equilibrium under more generic assumptions on $f$, namely that $f$ is unimodal function with a unique nontrivial equilibrium. In that work, the authors showed that the asymptotic stability can be obtained just from the analysis of the one-dimensional map associated to the equation, as requiring $f$ to have a negative Schwarzian derivative, which are independent of the kernel $\gamma$ involved, so they can be called \emph{absolute convergence conditions}; as well as other conditions related to particular forms of the kernel. We explain the link between \eqref{eq:prob} and \eqref{eq:Gurtin-MacCamy} in Section \ref{sec:gurtin}. In this work, we are not concerned with the global asymptotic stability of the stationary state in \eqref{eq:prob}, but with the existence of periodic solutions of maximal period, whose existence is suggested by the result of Ma and Magal \cite{MMa} at least in some cases.

Equation \eqref{eq:prob} is also an instance of nonlinear Volterra equation that are also interesting on their own. These equations, that can be written as  
\begin{equation}\label{eq:Volterra}
    b(t) = \int_0^{+\infty} \beta(a) f(b(t-s)) \dd s
\end{equation}
for a kernel $\beta$ normalized by $\int_0^{+\infty} \beta(s)\dd s=1$, have indeed attracted an independent interest, see e.g. \cite{asselah,Diekmann,CDM,DiekmannGils,Nussbaum} or \cite{N1979} for an integro-differential version; we also refer to \cite{MalletNussbaum, Magal-Ruan-2018} and the encyclopedia \cite{GripenbergLonden} for further references and historical notes.
\medskip

Slowly oscillating solutions to delay differential equations are particular solutions that have a minimal period that is somehow maximal with respect to the considered problem. This notions can be tracked back to the late '70s with the works of Kaplan and Yorke '75 \cite{Kaplan-Yorke-1975}, Nussbaum '79 \cite{Nussbaum-1979}, Walther '79 and '81\cite{Walther-1981}. For a fixed delay, slowly oscillating solutions are solutions with the property that the distance between two consecutive zeros is greater than the delay. 
For delay-differential equations that are more complex, this notions has to be adapted to the specific equation considered. In the context of delay differential equations with dynamic delays (the delay being a part of the differential equation), Arino, Hadeler and Hbid \cite{Arino} proposed the following definition: the distance between the left-hand side of two consecutive intervals on which the solution is identically zero, has to be no less than the maximal value that can be taken by the delay.  Such a strong definition would correspond in our case to solutions for which the distance between two zeros is no less than $1+\varepsilon/2$; however it can be proved that such solutions do not exist for \eqref{eq:prob}.
In this work we will call \textit{slowly oscillating solution} a solution that changes sign only once on any interval of size $1-\varepsilon/2$ (see the precise statement in Definition \ref{def:slowly-oscillating}).
Remark that the notion of slowly oscillating \textit{solution} should not be confused with that of slowly oscillating \textit{function}, which is a completely different notion independent of delay differential equations.

To the extent of our knowledge, the existence of slowly oscillating solutions for nonlinear Volterra equations such as \eqref{eq:prob} with negative feedback has only been established under strong symmetry assumptions on both the function $f$ and the solution. 
Nussbaum \cite{Nussbaum-1978} proved in 1978 that, under some conditions, nonlinear delay-differential equations with several delays may possess slowly oscillating solutions under the assumption that all nonlinearities are \textit{odd}; in our context, this amounts to assuming that the function $f$ in \eqref{eq:prob} is odd.    
Chow, Diekmann and Mallet-Paret \cite{CDM}  proved the existence and uniqueness of periodic solutions of fixed period 2, more precisely of sine-type solutions. 
More recently, Breda et al. \cite{Diekmann} showed, by means of numerical tools, the existence of a bifurcation in this type of renewal equations and also the existence of periodic solutions. Kennedy \cite{Kennedy} proved the existence of periodic solutions for a slight generalization of \eqref{eq:prob} under symmetric conditions for $f$.

In this work, we present a result of existence for a periodic, slowly oscillating  orbit of \eqref{eq:prob} without any symmetric condition over $f$ by means of a fixed-point-argument. Our argument is inspired by Arino, Hadeler and Hbid \cite{Arino}, although the method we use turns out to be quite different. The core of our argument is the study of the first-return or Poincaré map for the equation acting on an appropriate set of initial data that is consistent with slowly oscillating solutions. 
We chose to work with the set $\mathcal{B}^\alpha$ of functions over $[-1-\varepsilon/2, 0]$ that are bounded by a constant,  positive on a large interval and change sign exactly once between $-\varepsilon$ and $0$ in a non-degenerate manner (see the precise definition in Assumption \ref{as:Balpha}). 
This set shares some similarities with a cone but fails to be convex, which complicates the application of fixed-point theorems (in particular, we cannot apply the Schauder fixed-point Theorem directly).
We first prove that the Poincaré map is well-defined, continuous and leaves $\mathcal{B}^\alpha$ positively invariant by the Poincaré map. A key argument in our proof is the comparison with an eigenvector of the linearised equation of \eqref{eq:prob} far away from the zeros of the solution, which is made possible by the fact that the integral operator is monotone whenever the solution does not change sign in the whole domain of integration; and  when it does, a steepness argument shows that the solution is allowed to cross any eigenvector only once. 
In the process we show that  any initial data in $\mathcal{B}^\alpha$ yields a slowly oscillating solution of \eqref{eq:prob}. 
Then, we prove that our set $\mathcal{B}^\alpha$ possesses the fixed point property (see \cite{GD}) by constructing a homeomorphism to a convex subset of a different Banach space. This finishes the proof of existence of a slowly oscillating periodic solution of \eqref{eq:prob}.
\medskip

We also consider the singular limit of the constructed periodic orbit of \eqref{eq:prob} for a non-trivial subset of admissible $f$ and show that, when $\varepsilon\to 0$, it converges to a solution of the limit discrete-difference equation 
\begin{equation} \label{eq:difference}
    b(t) = f\big(b(t-1)\big) 
\end{equation}
that is slowly oscillating (see Ivanov and Sharkovsky \cite{Ivanov-Sharkovsky-1992} for detailed results on the solutions of \eqref{eq:difference}).

The paper is organized as follows. In Section \ref{sec:main-results} we present our main assumptions and results. Section \ref{sec:eigenvalue} gives a treatment of the linearized equation around 0 and the existence of the principal eigenvalue in the space of 2-periodic sine-type functions. In Section \ref{sec:space} we present a first result concerned to the norm-preserving property of the extension for any initial data. Section \ref{sec:Poincare} is devoted to the definition and main properties of the Poincaré map, such as its continuity and compactness. In section \ref{sec:Invariance} we present some a priori estimates for the continuation of an initial data in the phase space, as well as its connection with a convex set in a different Banach space. Section \ref{sec:proofs} is devoted to the proofs of our mains results. In Section \ref{sec:gurtin} we show how our result can be fitted in the setting of the Gurtin-MacCamy population model. Finally, in Section \ref{sec:simulations} we illustrate our results by means of numerical simulations where some phenomena can be observed, such as the Gibbs phenomenon in the non-monotone case.

\section{Main results}
\label{sec:main-results}

We start by our requirement concerning the function $f$ (see Figure  \ref{fig:1} for an illustration).
\begin{assumption}\label{as:f}
    $f\colon \R\to\R$ is a continuous function with $f'(0)<-1$, and satisfies the negative-feedback condition $xf(x)<0$. We let $\kappa_1$ and $\kappa_2$ be respectively the first negative and first positive solution of the equation $f(x)=-x$:
    \begin{align*}
	\kappa_1&:=\inf\{x<0\,:\, f(z)>-z, \quad  \,^\forall z\in [x, 0)\}, \\ 
	\kappa_2&:=\sup\{x>0\,:\, f(z)<-z,\quad  \,^\forall z\in (0, x]\}. 
    \end{align*}
    Furthermore, the following conditions hold
\[
    \limsup_{x\to +\infty} f(x)<0<\liminf_{x\to-\infty} f(x), \quad \text{and} \quad \liminf_{x\to \pm \infty} \frac{f(x)}{x} > -1.
\]
\end{assumption}

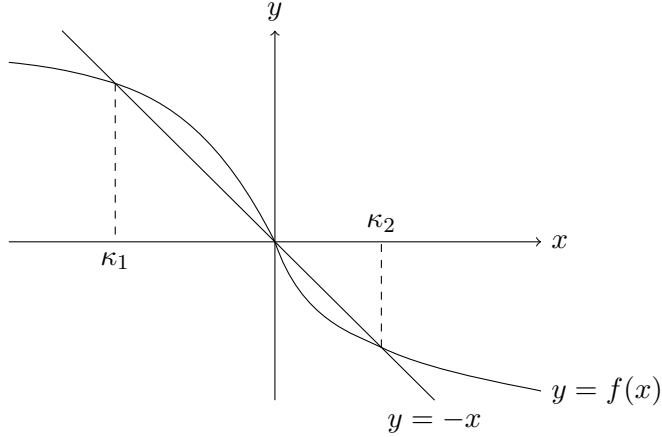
\begin{figure}[t]
    \centering
    \begin{tikzpicture}[scale=0.7]
    \draw[->] (-5,0) -- (5,0) node [anchor=west] {$x$};
    \draw[->] (0,-3) -- (0,4) node [anchor=south] {$y$};
    \draw (-4,4) -- (3,-3) node [anchor=north] {$y=-x$};
    \draw (-5,3.4) .. controls (-4,3.3) and (-3.5,3.15) .. (-3,3) .. controls (-2,2.67) and (-1,2) .. (0,0) .. controls (0.5,-1.5) and (1.5,-1.75) .. (2,-2) .. controls (3,-2.5) and (5,-2.8) .. (5,-2.83) node [anchor=west] {$y=f(x)$};
    \draw[dashed] (2,-2) -- (2,0) node [anchor=south] {$\kappa_2$};
    \draw[dashed] (-3,3) -- (-3,0) node [anchor=north] {$\kappa_1$};
    \end{tikzpicture}
    \caption{Graph of the function $f\colon \R \to \R$ with $\kappa_1<0<\kappa_2$.}
    \label{fig:1}
\end{figure}

Let us make precise the notion of solution that we will use in our study.  Given $b\in C([-1-\varepsilon/2,0])$ (the \textit{initial data}) it is possible to construct its continuation, namely $x_b\colon [-1-\varepsilon/2,+\infty) \to\R$, obeying the rule given by the equation, that is
\begin{equation}\label{eq:xb}
    x_b(t) = \begin{cases}
        b(t), & t\in[-1-\varepsilon/2, 0), \\
        \frac{1}{\varepsilon}\int_{1-\varepsilon/2}^{1+\varepsilon/2} f(x_b(t-s)) \dd s, & t\geq 0.
    \end{cases}
\end{equation}
Note that we do not impose the continuity of $x_b$ and in fact, $x_b(t)$ may possess a jump discontinuity at $t=0$. However, clearly, the restriction of $x_b$ to $[-1-\varepsilon/2, 0)$ and $[0, +\infty)$ are continuous. Moreover, $x_b$ is differentiable at all points $t\geq 0$ but $t\in\{0, 1-\varepsilon/2, 1+ \varepsilon/2 \}$, and the derivative is given by
\[
    x_b'(t) = \frac{1}{\varepsilon} [ f(x_b(t-1+\varepsilon/2)) - f(x_b(t-1-\varepsilon/2)) ].
\]

Next we define the notion of \textit{slowly oscillating solution} (see \cite[Definition 3.8]{Arino}).
\begin{definition}[Slowly oscillating solution]\label{def:slowly-oscillating}
    Let $t_0\in [-\infty, +\infty)$ be given and  $b:[t_0, +\infty)\to \mathbb{R}$ satisfy \eqref{eq:prob} for all $t\geq t_0+1+\frac{\varepsilon}{2}$. We will say that  $b(t)$ is a \textit{slowly oscillating solution} of \eqref{eq:prob} if the set $\{t\,:b(t)\neq 0\}$ can be written as a succession of nonempty open intervals, the distance between the left end of two consecutive intervals is no less than $1-\frac{\varepsilon}{2}$, and $b$ has different signs at each consecutive interval.  
\end{definition}

Before we can state our first result we need to describe the set of initial data that we will use in our main argument. First of all, consider the Banach space
\begin{align*}\label{def:C-0}
    C_0([a,b]) := \{ u\in C([a,b]) \colon u(a)=0 \}.
\end{align*}
We will define the set $\mathcal{B}^\alpha$ as  the set of continuous functions 
 $b\in C_0([-1-\varepsilon/2,0])$ 
 that satisfy the following assumption with $\varepsilon\in(0,1/4)$ and $\alpha>0$  given (since $\varepsilon$ is fixed we omit it in the notation of $\mathcal{B}^\alpha$). Importantly, we define
 \begin{equation} \label{eq:princ-eigvec}
     \varphi_0^\tau(t):=\sin\big(\pi(t+1+\tau)\big).
 \end{equation}
 The use of this function $\varphi_0$ as a lower barrier will be justified by the fact that $\varphi_0$ is a principal eigenvector of the linearised equation near $b(t)\equiv 0$.

Since the Assumption \ref{as:Balpha} is somewhat complex, we illustrate the two principal cases in Figure \ref{fig:B}.
\begin{assumption}\label{as:Balpha}
Let $b\in C_0([-1-\varepsilon/2,0])$. There exists a unique $\tau\in (0,\varepsilon]$ such that
    \begin{enumerate}[label={\roman*)}]
    \item $b(-\tau)=0$.
    \item The following bounds hold:
    \begin{align*}
        b(t)&\geq \gamma_\tau(t), \quad\forall t\in[-1-\varepsilon/2,-\tau], \\
        b(t)&\leq \alpha \varphi_0^{\tau}(t), \quad \forall t\in[-\tau,0],
    \end{align*}
    where $\varphi_0^\tau$ is defined by \eqref{eq:princ-eigvec} and the function $\gamma_\tau$ is a barrier of eigenfunctions depending on the values of $\varepsilon$ and $\tau$. More precisely, if $\tau\in(0,\varepsilon/2]$,
    \[
        \gamma_\tau(t) = \max_{\hat{\tau}\in[\tau,\varepsilon/2]} \alpha \varphi_0^{\hat{\tau}}(t);
    \]
    while if $\tau\in(\varepsilon/2,\varepsilon]$,
    \[
        \gamma_\tau(t) = \begin{cases}
            \displaystyle \max_{\hat{\tau}\in\{\tau,\varepsilon/2\}} \alpha\varphi_0^{\hat{\tau}}(t), & t\in[\tau^*-\varepsilon/2, \tau^*+ \varepsilon/2], \\
            \displaystyle (1-\Phi_{\tau^*-\varepsilon,\tau^*-\varepsilon/2}(t)) \varphi_0^{\varepsilon/2}(t) + \Phi_{\tau^*-\varepsilon,\tau^*-\varepsilon/2}(t) \varphi_0^\tau(t), & t\in(-1-\varepsilon/2,-\tau^*-\varepsilon/2), \\
            \displaystyle (1- \Phi_{\tau^*+\varepsilon/2,\tau^*+ \varepsilon}(t) ) \varphi_0^{\varepsilon/2}(t) + \Phi_{\tau^*+\varepsilon/2,\tau^* +\varepsilon}(t) \varphi_0^\tau(t), & t\in(\tau^* + \varepsilon/2, -\tau),
        \end{cases}
    \]
    where $\Phi_{a,b}(t)$, $a<b$, is a smooth non-decreasing function which satisfies
    \[
        0\leq \Phi_{a,b} \leq 1, \quad \Phi_{a,b}(t)=0, \quad t\leq a, \quad \Phi_{a,b}(t) = 1, \quad t\geq b;
    \]
     and $\tau^*$ is defined as the point at which $\varphi_0^{\varepsilon/2}$ and $\varphi_0^\tau$ intersects in $(-1-\varepsilon/2,-\tau)$, namely
    \[
        \tau^* = -\frac{1}{2}-\left( \frac{\varepsilon}{4} + \frac{\tau}{2} \right).
    \]
    \end{enumerate}
\end{assumption}
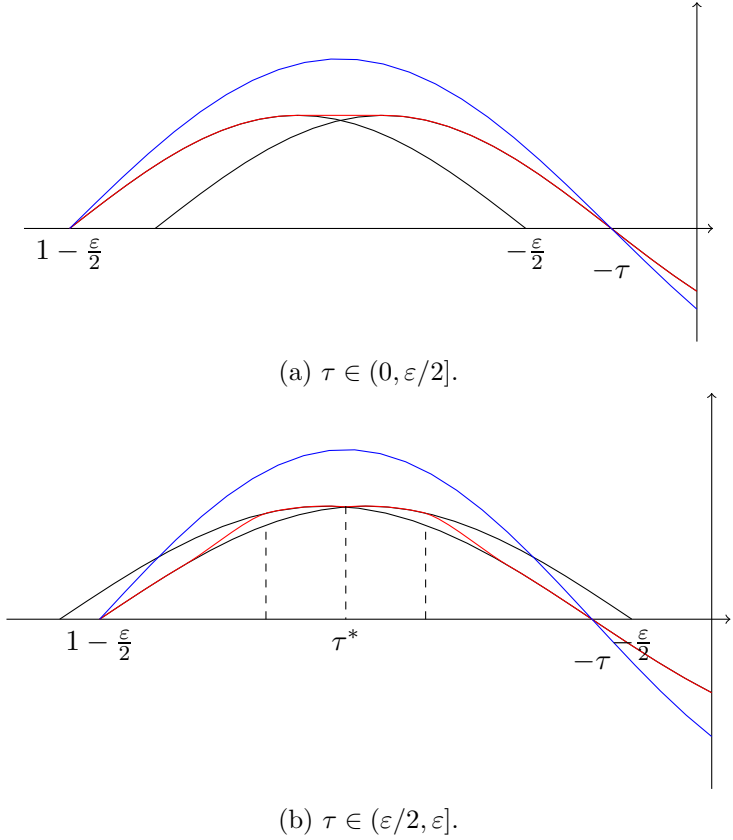
\begin{figure}[t]
    \centering
    \begin{subfigure}{\textwidth}
    \centering
    \begin{tikzpicture}[xscale=6, yscale=3]
        \foreach \e/\a in {0.75 / 0.5}
        {
            \draw[->] (-1-\e/2-0.1,0) -- (1pt,0);
            \draw[->] (0,-\a) -- (0,2*\a);
            \draw[domain=-1-\e/2:-\e/2] plot (\x, {\a*sin(pi*(1+\e/2+\x) r)});
            \draw[domain=-1-\e/4:0] plot (\x, {\a*sin(pi*(1+\e/4+\x) r)});
            \draw[color=red, domain=-1-\e/2:-1/2-\e/2] plot (\x, {\a*sin(pi*(1+\e/2+\x) r)});
            \draw[color=red] (-1/2-\e/2,\a) -- (-1/2-\e/4,\a);
            \draw[color=red, domain=-1/2-\e/4:0] plot (\x, {\a*sin(pi*(1+\e/4+\x) r)});
            \draw[color=blue,domain=-1-\e/2:0] plot (\x, {1.5*\a*sin(pi/(1+\e/4)*(1+\e/2+\x  ) r)});
            \draw (-1-\e/2,0) node [anchor=north] {$1-\frac{\varepsilon}{2}$};
            \draw (-\e/2,0) node [anchor=north] {$-\frac{\varepsilon}{2}$};
            \draw (-\e/4,-0.1) node [anchor=north] {$-\tau$};
        }
    \end{tikzpicture}
    \caption{$\tau\in(0,\varepsilon/2]$.}
    \end{subfigure}
    \\
    \begin{subfigure}{\textwidth}
    \centering
    \begin{tikzpicture}[xscale=7, yscale=3]
        \foreach \e/\a in {0.3 / 0.5}
        {
            \draw[->] (-1-3/4*\e-0.1,0) -- (1pt,0);
            \draw[->] (0,-1.5*\a) -- (0,2*\a);
            \draw[domain=-1-\e/2:-\e/2] plot (\x, {\a*sin(pi*(1+\e/2+\x) r)});
            \draw[domain=-1-3/4*\e:0] plot (\x, {\a*sin(pi*(1+3/4*\e+\x) r)});
            \draw[color=red, domain=-1-\e/2:-1/2-13/8*\e] plot (\x, {\a*sin(pi*(1+\e/2+\x) r)});
            \draw[color=red] (-1/2-13/8*\e,0.49*\a) .. controls (-1/2-12/8*\e,0.59*\a) and (-1/2-10/8*\e,0.89*\a) .. (-1/2-9/8*\e,0.935*\a);
            \draw[color=red, domain=-1/2-9/8*\e:-1/2-5/8*\e] plot (\x, {\a*sin(pi*(1+3/4*\e+\x) r)});
            \draw[color=red, domain=-1/2-5/8*\e:-1/2-1/8*\e] plot (\x, {\a*sin(pi*(1+\e/2+\x) r)});
            \draw[color=red] (-1/2-1/8*\e,0.935*\a) .. controls (-1/2,0.89*\a) and (-1/2+2/8*\e,0.59*\a) .. (-1/2+3/8*\e,0.49*\a);
            \draw[color=red, domain=-1/2 + 3/8*\e:0] plot (\x, {\a*sin(pi*(1+3/4*\e+\x) r)});
            \draw[dashed] (-1/2-5/8*\e, \a) -- (-1/2-5/8*\e,0) node [anchor=north] {$\tau^*$};
            \draw[color=blue, domain=-1-\e/2:0] plot (\x, {1.5*\a*sin(pi/(1-\e/4)*(1+\e/2+\x  ) r)});
            \draw[dashed] (-1/2-9/8*\e,0) -- (-1/2-9/8*\e,0.8*\a);
            \draw[dashed] (-1/2-1/8*\e,0) -- (-1/2-1/8*\e, 0.8*\a);
            \draw (-1-\e/2,0) node [anchor=north] {$1-\frac{\varepsilon}{2}$};
            \draw (-\e/2,0) node [anchor=north] {$-\frac{\varepsilon}{2}$};
            \draw (-3/4*\e,-0.1) node [anchor=north] {$-\tau$};
        }
    \end{tikzpicture}
    \caption{$\tau\in(\varepsilon/2,\varepsilon]$.}
    \end{subfigure}
    \caption{Typical shape of the functions (in blue) which belong to $\B^\alpha$. The barrier of eigenfunction is plotted in red. (a) for $\tau\in(0, \varepsilon/2]$. (b) for $\tau\in (\varepsilon/2, \varepsilon]$.}
    \label{fig:B}
\end{figure}

Now we define $\B^\alpha = \{ b\in C_0([-1-\varepsilon/2,0]) \colon \text{Assumption \ref{as:Balpha} holds} \}$. That is, the functions which belong to this set maintain above and below some barrier of shifts of the eigenfunction for the linear problem, $\varphi_0$. We provide an illustration in Figure \ref{fig:B}.
\begin{remark}
Even though the definition of $\gamma_\tau$ for $\tau\in(\varepsilon/2,\varepsilon]$ is quite complex, we mainly use the following property: 
\[
    \gamma_\tau(t) \geq \begin{cases}
    \displaystyle
    \min_{\hat{\tau}\in\{\tau,\varepsilon/2\}} \alpha\varphi_0^{\hat{\tau}}(t), & t\in(-1-\varepsilon/2,-\tau)\setminus[\tau^*-\varepsilon/2,\tau^*+\varepsilon/2], \\
    \displaystyle \max_{\hat{\tau}\in\{\tau,\varepsilon/2\}} \alpha\varphi_0^{\hat{\tau}}(t), & t\in[\tau^*-\varepsilon/2,\tau^*+\varepsilon/2].
    \end{cases}
\]
Actually, our definition of $\gamma_\tau$ aims at correcting the discontinuity of the function defined on the right-hand side of the above inequality.
\end{remark}

We are now in a position to state our main results.  The first two concern the existence of slowly oscillating solutions of \eqref{eq:prob}. 
\begin{theorem}\label{thm:slowly-oscillating-extension}
	Let $\varepsilon\in (0, 1/4)$ and $\alpha>0$ be given. If $b\in \mathcal{B}^\alpha$ then its continuation $x_b(t)$, defined by \eqref{eq:xb}, is slowly oscillating. 
\end{theorem}
The proof of Theorem $\ref{thm:slowly-oscillating-extension}$ is based on establishing bounds for the positive zeros of the continuation $x_b$ and check that at each of these points it changes its sign.
\begin{theorem}\label{thm:existence-solutions}
    Let $f$ satisfy Assumption \ref{as:f} and suppose that $f'(0)<-2$.  Let $\varepsilon_0>0$ be the smallest solution of the equation 
    \begin{equation*}
	\dfrac{\sin\left(\frac{\pi\varepsilon}{2}\right)}{\frac{\pi\varepsilon}{2}} = -\frac{2}{f'(0)}. 
    \end{equation*}
    For each $\varepsilon\leq\varepsilon_0$, there exists a periodic, slowly oscillating solution of \eqref{eq:prob}. 
\end{theorem}
The main idea behind Theorem \ref{thm:existence-solutions} is to study the Poincaré map of \eqref{eq:prob} defined as the shift of $b$ to its first positive zero; more precisely, if $z_1(b)$ (which we define later, in \eqref{eq:z_1}) is the first time $t\geq 0$ such that $x_b(t)=0$,  then  $\mathcal{F}(b)(t)=x_b(t+z_1(b)+1+\varepsilon/2)$. It turns out that we can prove that $\mathcal{B}^\alpha $ is stable by $\mathcal{F}$.

Finally we  consider the behavior of the slowly oscillating periodic solution in the singular limit $\varepsilon\to 0$.
\begin{theorem}\label{thm:limit-behavior}
    Let $f$ be a decreasing function which satisfies Assumption \ref{as:f} and suppose that $f'(0)<-2$. Assume moreover that $\kappa_2=-\kappa_1=:\kappa_0$, and $f'(\kappa_0)>-1$ and $f'(\kappa_0)<-1$.  Let $b^{\varepsilon}(t)$ be a family of periodic, slowly oscillating solutions of \eqref{eq:prob} with minimal period, bounded in the sup-norm by $\kappa_0$, and such that the restriction of $b^\varepsilon$ to $[-1-\varepsilon/2, 0]$ belongs to $\mathcal{B}^\alpha$. Then $b^\varepsilon$ converges as $\varepsilon\to 0$ to a square wave $b^*(t)$ solution of $\varphi(t)=f\big(\varphi(t-1)\big)$ and that satisfies 
	\begin{equation*}
		b^*(t)= 
		\begin{cases}
			\kappa_0 & \text{ if }  \lfloor t\rfloor \text{ is odd },\\
			-\kappa_0 & \text{ if } \lfloor t\rfloor \text{ is even},
			\end{cases}
	\end{equation*}
	where $\lfloor t\rfloor$ denotes the integer part of $t$.
	The convergence holds uniformly in any closed interval  $I$ with $I\cap\mathbb{Z}=\varnothing$ and in  $L^1_{loc}(\mathbb{R})$.
\end{theorem}

\section{The linearised equation}\label{sec:eigenvalue}

Linearising the problem \eqref{eq:prob} around the trivial equilibrium 0 we obtain
\begin{equation}\label{eq:linear}
    u(t) = \frac{f'(0)}{\varepsilon}\int_{1-\varepsilon/2}^{1+\varepsilon/2} u(t-s)\dd s.
\end{equation}
Let $E$ be the Banach space of odd sine-type functions endowed with the sup-norm, that is, $E=\{ u\in C(\R) \colon u(-t)=-u(t)~\text{and}~ u(t+1)=-u(t) \}$, and $T$ be the operator given by
\begin{equation}\label{eq:op_linear}
    (Tu)(t) = -\frac{1}{\varepsilon}\int_{1-\varepsilon/2}^{1+\varepsilon/2} u(t-s)\dd s.
\end{equation}

It is easy to check that $E$ is invariant under $T$, and then $T\colon E \to E$. Let us define the cone $P=\{ u\in E \colon u(t)\geq0~\forall t\in[0,1] \}$. In the rest of the paragraph we establish the existence and uniqueness of a positive eigenfunction for $T$, and the simplicity of the corresponding positive eigenvalue.

In a similar context, Chow, Diekmann and Mallet-Paret in \cite[Proposition 3.4]{CDM}  proved that the iterations of $T$ {\it increase} the support of a non-zero function in the cone $P$. Here, by replicating their argument, we can obtain the  following result.
\begin{prop}\label{prop:positivity}
    For every function $u\in P\setminus \{0\}$ there exists $n\in\N$ such that $(T^n u)(t)>0$ for all $t\in(0,1)$.
\end{prop}
As the proof does not bring any new argument compared to \cite[Proposition 3.4]{CDM}, we omit it for the sake of concision.

Next we state some terminology about linear positive operator. For more results in this topic see \cite{MAK}.

\begin{defi}
    Let $E$ be a Banach space with a cone $K$ and $u_0\in K\setminus\{0\}$. A linear operator $A\colon E\to E$ is said to be $u_0$-\emph{positive} if for every $x\in K\setminus\{0\}$ there exists $n\in\N$ and positive numbers $\alpha,\beta$ such that
    \[
        \alpha u_0 \leq A^n x \leq \beta u_0.
    \]
\end{defi}

With this definition it is possible to establish that the operator $T$ is, actually, $u_0$-positive for the function $u_0\in P$ represented in Figure \ref{fig:2}.

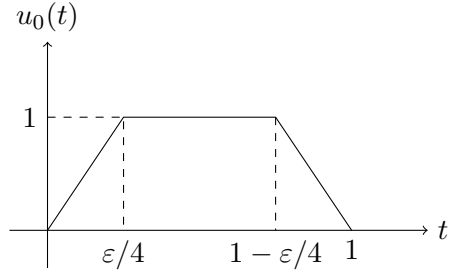
\begin{figure}[h]
    \centering
    \begin{tikzpicture}
    \draw[->] (-0.5,0) -- (5,0) node [anchor=west] {$t$};
    \draw[->] (0,-0.5) -- (0,2.5) node [anchor=south] {$u_0(t)$};
    \draw (0,0) -- (1,1.5) -- (3,1.5) -- (4,0) node [anchor=north] {1};
    \draw[dashed] (0,1.5) node [anchor=east]{1} -- (1,1.5) -- (1,0) node [anchor=north]{$\varepsilon/4$};
    \draw[dashed] (3,1.5) -- (3,0) node [anchor=north]{$1-\varepsilon/4$};
    \end{tikzpicture}
    \caption{Graph of the function $u_0\colon \R \to \R$ for $t\in[0,1]$.}
    \label{fig:2}
\end{figure}

Then we have the following result.
\begin{lem}\label{lem:u0-positive}
    The operator $T\colon E \to E$ defined by \eqref{eq:op_linear} is $u_0$-positive.
\end{lem}
\begin{proof}
Due to the symmetry of the functions considered, the operator $T$ can be rewritten as
\[
    (Tu)(t) = \frac{1}{\varepsilon} \int_{-\varepsilon/2}^{\varepsilon/2} u(t-s) \dd s = \frac{1}{\varepsilon} \int_{t-\varepsilon/2}^{t+\varepsilon/2} u(s)\dd s,
\]
and therefore $(Tu)'(t) = \varepsilon^{-1}[ u(t+\varepsilon/2) - u(t-\varepsilon/2) ]$ for all $t\in \R$. In particular, this implies that $Tz$ is strictly increasing on $[-\varepsilon/2,\varepsilon/2]$ and strictly decreasing on $[1-\varepsilon/2, 1+\varepsilon/2]$ if $u(t)>0$ in $(0,1)$.

Let $u\in P\setminus\{0\}$ and $n\in \N$ be the number given by Proposition \ref{prop:positivity}. Note that $T^{n+1}u$ is strictly increasing on $[-\varepsilon/2, \varepsilon/2]$ and decreasing on $[1-\varepsilon/2, 1+\varepsilon/2]$ by the previous observation. Actually, there exists constants $K,k>0$ such that
\[
    k \leq|(T^{n+1}u)'(t)| \leq K, \quad t\in[0,\varepsilon/2]\cup [1-\varepsilon/2,1].
\]
    Since $T^{n+1}u$ is positive on $(0,1)$, it is possible to take $\alpha$ and $\beta$ positive numbers small and large enough, respectively, in a manner that $\alpha u_0 \leq T^{n+1} u \leq \beta u_0$, hence completing the proof of Lemma \ref{lem:u0-positive}.
\end{proof}

Thanks to the previous results, we can state the existence and uniqueness of a positive eigenvalue and a positive eigenfunction.

\begin{thm}
	The operator $T\colon E \to E$ possesses a unique eigenfunction $\varphi_0$ (up to multiplication by a positive scalar) which belongs to $P$.
\end{thm}
\begin{proof}
We first prove that the operator has a positive eigenfunction. By simple inspection it can be seen that $\varphi_0(t) = \sin(\pi t)$ is a positive eigenfunction for $T$. In fact,
\[
    (T\varphi_0)(t) = \frac{2}{\pi\varepsilon} \sin\left( \frac{\pi\varepsilon}{2} \right)\varphi_0(t).
\]
Then, the results follows immediately from Theorems 2.10 and 2.11 in \cite{MAK}, and the previous lemma.
\end{proof}

For the rest of the work, let us fix some notation. We denote $\varphi_0(t) =\sin(\pi t)$ the unique function such that $\varphi_0\in P$ (recall \eqref{eq:princ-eigvec}) and 
\begin{equation}\label{eq:princ-eigval}
    \lambda_0 = -f'(0)\frac{2}{\pi\varepsilon}\sin\left( \frac{\pi\varepsilon}{2}  \right),
\end{equation}
the positive number such that $M\varphi_0 = \lambda_0 \varphi_0$, where $M = -f'(0)T$. That is, $M$ corresponds to the operator given by the linearised equation \eqref{eq:linear}. Additionally, for $\tau\in\R$, we set the shifted function $\varphi_0^\tau(t) = \varphi_0(t+1+\tau)$.

\section{A first upper bound}\label{sec:space}

Looking at the definition of $x_b$ in \eqref{eq:xb}, we suspect its norm is controlled by the properties of $f$. In particular, by Assumption \ref{as:f} there exists a positive constant $a_0$ (as small as we want) such that $f$ is strictly decreasing on $[-a_0,a_0]$ and
\begin{equation}\label{eventual_posi}
    |f(x)| \geq \max\{f(-a_0), |f(a_0)|\} = \max_{y\in[-a_0,a_0]} |f(y)|, \quad \,^\forall |x|\geq a_0,
\end{equation}
and also a constant $A_0>0$ such that 
\begin{equation}\label{diag_cond}
    |f(x)| < |x|, \quad \,^\forall |x|\geq A_0.
\end{equation}

Now, the aforementioned property  can be stated formally as follows.

\begin{prop}\label{unif_bound}
	Let $R\geq \max_{|x|\leq A_0}|f(x)|$. Then, for every continuous function $b\colon [-1-\varepsilon/2,0]\to \R$ with $\|b\|\leq R$, it follows $|x_b(t)|\leq R$ for all $t\geq 0$ (here, $\|\cdot\|$ denotes the sup-norm).
\end{prop}
\begin{proof}
Let $R$ as in the statement and $b$ be such a continuous function. Then, for $t\in[0,1-\varepsilon/2]$,
\[
    |x_b(t)| \leq \frac{1}{\varepsilon}\int_{t-1-\varepsilon/2}^{t-1+\varepsilon/2} |f(b(s))| \dd s \leq \max_{x\in[-R,R]} |f(x)|.
\]
If $R\leq A_0$ then clearly $|x_b(t)|\leq R$. On the other hand, if $R>A_0$, by \eqref{diag_cond} we deduce $|f(x)|\leq R$ for $A_0\leq|x|\leq R$ and therefore $|x_b(t)|\leq R$. Consequently, we have proved the result for $t\in[0,1-\varepsilon/2]$, and the general bound follows by induction.
\end{proof}

\section{The Poincaré map}\label{sec:Poincare}

\subsection{Assumptions and definition}

For any $b\in \mathcal{B}^\alpha$, we define 
\begin{equation}\label{eq:z_1}
	z_1(b):= \sup\{ t\geq 0 \colon x_b(s)<0~\forall s\in [0,t] \},
\end{equation}
which corresponds to the first positive zero of $x_b$. Observe $z_1(b)$ is well-defined and finite. Indeed  $0$ satisfies $x_b(0)<0$ and therefore $z_1(b)>-\infty$. Let us show by contradiction that $z_1(b)<+\infty$. If $z_1(b)=+\infty$ then $x_b(t)<0$ for all $t\geq0$; but then
\[
    x_b(1+\varepsilon/2) = \frac{1}{\varepsilon} \int_{0}^{\varepsilon} f(x_b(s)) \dd s>0
\]
owing to the negative-feedback condition, which is a contradiction. Therefore, $z_1(b)$ is well defined, and is possible to define $z_j(b)$ as the $j^{\text{th}}$-positive zero of $x_b$ in a similar fashion.

\begin{remark}\label{rem:boundontau}
Observe that $z_1(b) \in (1-\tau-\varepsilon/2,1-\tau+\varepsilon/2)$. Indeed,
\[
    x_b(1-\tau-\varepsilon/2) = \frac{1}{\varepsilon} \int_{-\tau-\varepsilon}^{-\tau} f(b(s))\dd s <0,
\]
which implies that $x_b(t)<0$ for all $t\in[0,1-\tau-\varepsilon/2]$ and thus $z_1(b)> 1-\tau-\varepsilon/2$. On the other hand,
\[
    x_b(1-\tau+\varepsilon/2) = \frac{1}{\varepsilon} \int_{-\tau}^{-\tau+\varepsilon} f(x_b(s))\dd s >0,
\]
as $-\tau + \varepsilon < 1-\tau-\varepsilon/2$, and therefore $z_1(b)<1-\tau+\varepsilon/2$.
\end{remark}

\begin{prop}\label{deriv_x}
    Let $b\in \mathcal{B}^\alpha$ be given. If $x_b$ is differentiable at $z_1(b)$, then $x'_b(z_1(b))>0$. If $x_b$ is not differentiable at that point, the two half-derivatives 
	\begin{equation*}
		\partial_+x_b(t):=\lim_{h\to 0, h>0} \dfrac{x_b\big(t+h\big)-x_b\big(t\big)}{h} \text{ and } \partial_- x_b(t) := \lim_{h\to 0, h<0} \dfrac{x_b\big(t+h\big)-x_b\big(t\big)}{h} 
	\end{equation*}
	exist at $t=z_1(b)$, and both are positive:
    \begin{equation*}
	    \partial_+ x_b\big(z_1(b)\big)>0 \text{ and } \partial_- x_b\big(z_1(b)\big)>0.
    \end{equation*}
\end{prop}
\begin{proof}
	Since $b$ is a continuous function on $[-1-\varepsilon/2, 0]$, there are exactly two points where the derivative of $x_b(t)$ might be discontinuous, that are $t=1-\frac{\varepsilon}{2}$ and $t=1 +\frac{\varepsilon}{2}$. Since $z_1(b)< 1-\tau+\frac{\varepsilon}{2}\leq 1+\frac{\varepsilon}{2}$ (see Remark \ref{rem:boundontau}), we have only two cases: either $x_b$ is differentiable at $z_1(b)$, or $z_1(b)=1-\frac{\varepsilon}{2}$. 

	\noindent $\bullet$
If $x_b$ is differentiable at $z_1(b)$, it is clear that
\[
    x_b'(z_1(b)) = \frac{1}{\varepsilon}[ f(x_b(z_1(b)-1+\varepsilon/2)) - f(x_b(z_1(b)-1-\varepsilon/2)) ]>0
\]
	as $z_1(b)-1+\varepsilon/2\in(-\tau,z_1(b))$ and $z_1(b)-1-\varepsilon/2 < -\tau$. 

	\noindent$\bullet$ If $x_b$ is not differentiable at this point, then necessarily $z_1(b)=1-\varepsilon/2$ and we have
	\begin{align*}
		\partial_+ x_b\big(z_1(b)\big) &= \lim_{h\to 0^+} \frac{1}{\varepsilon}\left[\frac{1}{h}\int_0^h f\big(x_b(s)\big)\dd s-\frac{1}{h}\int_{{-\varepsilon}}^{-\varepsilon+h}f\big(b(s)\big)\dd s\right] \\
		&=\frac{1}{\varepsilon} \left[\lim_{h\to 0^+}f\big(x_b(h)\big) - f\left(b\left(-\varepsilon\right)\right)\right] \\
		&= \frac{1}{\varepsilon}\left[f\left(\frac{1}{\varepsilon}\int_{-1-\varepsilon/2}^{-1+\varepsilon/2} f\big(b(s)\big)\dd s \right)-f\left(b\left(-\varepsilon\right)\right)\right], 
	\end{align*}
	and clearly $f\left(\frac{1}{\varepsilon}\int_{-1-\varepsilon/2}^{-1+\varepsilon/2} f\big(b(s)\big)\dd s \right)>0$ since $b(s)>0$ on $(-1-\varepsilon/2, -1+\varepsilon/2)$, while $f\left(b\left(-\varepsilon\right)\right)<0$. Similarly, 
	\begin{equation*}
		\partial_- x_b\big(z_1(b)\big) = \frac{1}{\varepsilon}\left[f\big(b(0)\big)-f(b(-\varepsilon))\right]>0. 
	\end{equation*}
	This finishes the proof of Proposition \ref{deriv_x}
\end{proof}

Next, we can define the following operator
\begin{equation}\label{eq:poincare_operator}
    \begin{array}{ccl}
         \Fcal \colon \B^{\alpha} \subseteq C_0([-1-\varepsilon/2,0])& \longrightarrow & C_0([-1-\varepsilon/2,0])  \\
         b& \longmapsto &x_b(z_1(b) +1+\varepsilon/2+\cdot)=: \Fcal (b),
    \end{array}
\end{equation}
which is, actually, the Poincaré map, or first-return map, associated to \eqref{eq:prob}.

\subsection{Continuity and compactness of the Poincaré map}

This paragraph is devoted to the continuity of the Poincaré operator. We first establish some properties related to the continuous dependence of $x_b$ on the initial data $b$. More precisely the next lemma holds

\begin{lem}\label{lem31}
    Let $R\geq \max_{|x|\leq A_0}|f(x)|$. For every $T>0$ and $\eta>0$, there exists $\delta>0$ such that if $b,\psi\in \B^{\alpha}$, $\|b\|,\|\psi\|\leq R$, with $\|b-\psi\|<\delta$ then
    \[
        |x_b(t) - x_\psi(t)| < \eta, \quad \forall~t\in[0,T].
    \]
\end{lem}
\begin{proof}
We will prove that for every $n\in\N$ and $\eta>0$ there exists $\delta>0$ such that if $\|b-\psi\| < \delta$ then 
\[
    |x_b(t)-x_\psi(t)|<\eta, \quad \forall~t\in[0,n(1-\varepsilon/2)],
\]
by the means of induction. For the case $n=1$, as $f$ is uniformly continuous on $[-R,R]$, there exists $\delta>0$ such that if $|x-y|<\delta$ with $x,y\in[-R,R]$ then
\[
        |f(x)-f(y)|<\eta.
\]
Thus, setting $\|b-\psi\| < \delta$ we obtain
\begin{align}\label{eq:estim_continuity}
    |x_b(t)-x_\psi(t)| &\leq \frac{1}{\varepsilon} \int_{t-1-\varepsilon/2}^{t-1+\varepsilon/2} |f(b(s))-f(\psi(s))| \dd s \leq \eta,
\end{align}
for all $t\in[0,1-\varepsilon/2]$.

Next, let us assume that the result is true for certain $n\in\N$ and we will prove it for $n+1$. So, from the induction hypothesis, we deduce the existence of $\overline{\delta}\in (0,\delta)$ such that if $\|b-\psi\| < \overline{\delta}$ then
\[
    |x_b(t) - x_\psi(t)| < \min\{\delta,\eta\} \quad \forall~ t\in[0,n(1-\varepsilon/2)],
\]
where $\delta$ is as before. With this choice, as $|x_b(t)|,|x_\psi(t)|\leq R$ for $t\in[0,n(1-\varepsilon/2)]$ (see Proposition \ref{unif_bound}) then
\[
    |x_b(t) - x_\psi(t)| \leq \eta,
\]
for all $t\in[n(1-\varepsilon/2, (n+1)(1-\varepsilon/2)]$, as in \eqref{eq:estim_continuity}.
\end{proof}

\begin{thm}\label{teo:11}
    The map $b \mapsto z_1(b)$ is continuous from $\B^{\alpha}$ into $\R$.
\end{thm}
\begin{proof}
Let $b\in \B^{\alpha}$ and $\mu>0$. As a consequence of Proposition \ref{deriv_x}, there  exists $r\in(0,\mu/2)$ such that $x_b(z_1(b)+r)>0$. Noting that
\[
    0< \min\left\{ \min_{[0,z_1(b)-r]} -x_b(t), x_b(z_1(b)+r) \right\},
\]
it is possible to apply the previous lemma in the interval $[0,z_1(b)+\mu]$ and for a suitable $R$ (since $b$ is fixed), thus there exists $\delta>0$ such that, if $\|b-\psi\| < \delta$,
\[
    x_\psi(t)<0 \quad \forall ~t\in[0,z_1(b)-r], \quad\text{and}\quad x_\psi(z_1(b)+r)>0,
\]
thus $z_1(\psi)\in(z_1(b)-r,z_1(b)+r)$, and therefore $|z_1(\psi) - z_1(b)|<\mu$.
\end{proof}

\begin{lem}\label{lem21}
    Let $R\geq \max_{|x|\leq A_0} |f(x)|$. For every  $b\in \B^\alpha$, $T>0$ and $\eta>0$ there exists $\delta>0$ such that if $b,\psi\in \B^\alpha$, $\|b\|,\|\psi\|\leq R$, with $\|b-\psi\| < \delta$ then
    \[
        |x_b(t) - x_\psi(s)| < \frac{2 R}{\varepsilon}|t-s| + \eta \quad\forall~t,s\in[0,T]~\text{with}~ |t-s|<\varepsilon.
    \]
\end{lem}
\begin{proof}
Observe that if $|t-s|<\varepsilon$ and, without loss of generality, if we assume that $t<s$ then $s-\varepsilon/2<t+\varepsilon/2$. Thus, taking $\delta>0$ given by Lemma \ref{lem31}, we have
\begin{align*}
    |x_b(t)-x_\psi(s)| &\leq |x_b(t) - x_b(s)| + |x_b(s) - x_\psi(s)| \\
    & < \frac{1}{\varepsilon} \left| \int_{t-1-\varepsilon/2}^{s-1-\varepsilon/2} f(x_b(y))~dy - \int_{t-1+\varepsilon/2}^{s-1+\varepsilon/2} f(x_b(y))~dy \right| +\eta \\
    & \leq \frac{2(s-t)}{\varepsilon} \max_{|x|\leq R}|f(x)| + \eta \\
    &\leq \frac{2R}{\varepsilon}(s-t) + \eta,
\end{align*}
for $\|b-\psi\|< \delta$.
\end{proof}

\begin{thm}
    The map $b\mapsto \Fcal (b)$ is continuous from $\B^{\alpha}$ into $C_0([-1-\varepsilon/2,0])$.
\end{thm}
\begin{proof}
Let $b\in\B^{\alpha}$ and $\eta>0$. Using the continuity of the map $b\to z_1(b)$ we know there exists $\delta>0$ such that if $\|b-\psi\|<\delta$ then
\[
|z_1(b)- z_1(\psi)|<\min \left\{ \varepsilon, \frac{\eta\varepsilon}{4R} \right\},
\]
where $R>0$ is a suitable fixed constant which satisfies $R> \|b\|$ and $R\geq \max_{|x|\leq A_0} |f(x)|$.

Moreover, using the previous lemma with, and taking $\delta$ smaller if needed, it follows that for $\|b - \psi\|<\delta$ and $t\in[-1-\varepsilon/2,0]$, then
\begin{align*}
    |\Fcal(b)(t) - \Fcal(\psi)(t)| & = |x_b(z_1(b)+1+\varepsilon/2 +t) - x_\psi(z_1(\psi) + 1 +\varepsilon/2 + t)| \\
    &< \frac{2R}{\varepsilon}|z_1(b) - z_1(\psi)| + \frac{\eta}{2} \\
    &< \eta,
\end{align*}
which shows that $\|\Fcal(b) - \Fcal(\psi)\| \leq\eta$ and therefore $\Fcal$ is continuous.
\end{proof}

Finally, another important property of the map $\Fcal$ is its compactness.

\begin{thm}
	The set $\mathcal{F}(\B^\alpha)$ is compact for the topology of $C_0([-1-\varepsilon/2, 0])$.  
\end{thm}
\begin{proof}
Let $(b_n)_n\subseteq \B^{\alpha}$ be a bounded sequence, say by $R\geq \max_{|x|\leq A_0}|f(x)|$. Observe that the sequence $(\Fcal(b_n))_n$ has almost everywhere uniformly bounded derivatives, that is, all but the points at which its derivative does not exists; indeed,
\[
    |\Fcal(b_n)'(t)| = \frac{1}{\varepsilon}| f(x_{b_n}(t-1+\varepsilon/2)) - f(x_{b_n}(t-1-\varepsilon/2)) | \leq \frac{2}{\varepsilon} \sup_{x\in[-R,R]} |f(x)|\leq \frac{2R}{\varepsilon}.
\]

This shows that $(\Fcal(b_n))_n$ is an equicontinuous and bounded family in $C_0([-1-\varepsilon/2,0])$, and therefore it possesses a uniformly convergent subsequence by the Arzelà-Ascoli Theorem.  Thus $\mathcal{F}(\B^\alpha)$ is compact.
\end{proof}

\section{Homeomorphisms and Invariance}\label{sec:Invariance}

\subsection{A priori estimates for the extensions}

By definition, for every $\delta_0>0$ there exists an $\alpha_0>0$ such that
\[
    \frac{f(x)}{x} \leq f'(0) + \delta_0 <0 \quad \forall |x|\leq \alpha_0.
\]
Hereinafter, we assume that $f'(0)<-2$, and $\delta_0,\varepsilon,\alpha$ are sufficiently small so that the following is verified:
\begin{itemize}
    \item $\lambda_0\left( 1+\frac{\delta_0}{f'(0)} \right) \geq 2$.
    \item $\alpha\lambda_0\leq \min\{\alpha_0,a_0\}$, where $\alpha_0$ is as in the previous note and $a_0$ is as in \eqref{eventual_posi}.
\end{itemize}

\begin{remark}
The conditions imposed over $\varepsilon$ and $\delta_0$ also implies that
\[
    \lambda_0\left( 1+\frac{\delta_0}{f'(0)} \right) \cos\left( \frac{\pi\varepsilon}{2} \right)>1.
\]
Indeed, as $\varepsilon\leq 1/4$ then $\cos(\pi\varepsilon/2)\geq \cos(\pi/8)>1/2$. In general, this holds for every $\varepsilon<2/3$.
\end{remark}

\begin{remark}
With this choice for $\alpha$, every continuous function $b\colon [0,2]\to \R$ for which $\frac{b(t)}{\alpha\lambda_0 \varphi_0(t)}\geq 1$ for $t\in(0,2)\setminus\{1\}$, also satisfies $\frac{f(b(t))}{f(\alpha\lambda_0 \varphi_0(t))}\geq 1$ for $t\in(0,2)\setminus\{1\}$.

Indeed, let $t\in(0,1)$, then we have $b(t)\geq \alpha\lambda_0\varphi_0(t)>0$. If $b(t)\geq a_0$ then $|f(b(t))|\geq |f(\alpha\lambda_0\varphi_0(t))|$ (see \eqref{eventual_posi}), and from the negative-feedback condition we get the claim. Reciprocally, if $b(t)\leq a_0$ then the statement follows from the monotonicity and negative-feedback condition on $f$. An analogous argument proves the statement for $t\in(1,2)$.
\end{remark}

The previous remark justifies the manipulations done in the next proposition. Mainly, it allows us to establish the order relation between $f(x_b(t))$ and $f(\alpha\lambda_0 \varphi_0^\tau(t))$.

\begin{prop}[Estimates for $\tau\in(0,\varepsilon/2{]}$]\label{prop:estimates1}
Let $b\in\B^\alpha$ with $\tau\in(0,\varepsilon/2]$. Then, the extension $x_b$ verifies the following estimates:
\begin{align}
    x_b(t) &\leq \lambda_0 \left( 1+\frac{\delta_0}{f'(0)} \right) \min_{\hat{\tau}\in [\tau,\varepsilon/2]} \alpha \varphi_0^{\hat{\tau}} (t), \quad t\in[0,1-\tau-\varepsilon/2]; \label{estimate1} \\ 
    \frac{x_b(t)}{\alpha \varphi_0^{1-z_1(b)}(t)} &>1, \quad t\in(1-\tau-\varepsilon/2,1-\tau+\varepsilon/2) \setminus \{ z_1(b) \}; \label{stepness} \\     
    x_b(t) &\geq \alpha\lambda_0 \left( 1+\frac{\delta_0}{f'(0)} \right) \varphi_0^\tau (t), \quad t\in[1-\tau+\varepsilon/2, 2-\tau-\varepsilon];\label{estimate2} \\
    x_b(t) &\geq \alpha\lambda_0 \left( 1+\frac{\delta_0}{f'(0)} \right)\varphi_0^{1-z_1(b)}(t), \quad t\in(2-\tau-\varepsilon,z_1(b)+1-\varepsilon/2]; \label{estimate3} \\ 
    \frac{x_b(t)}{\alpha \varphi_0^{-z_2(b)}(t)} &>1, \quad t\in(z_1(b)+1-\varepsilon/2,z_1(b)+1+\varepsilon/2) \setminus\{ z_2(b) \} \label{stepness2}.
\end{align}
\end{prop}
\begin{proof}
To prove \eqref{estimate1} let $t\in [0,1-\tau-\varepsilon/2]$. Using Assumption \ref{as:f} and the previous remark we obtain
\begin{align*}
    x_b(t) &= \frac{1}{\varepsilon}\int_{t-1-\varepsilon/2}^{t-1+\varepsilon/2} f(b(s)) \\
    &\leq \frac{1}{\varepsilon}\int_{t-1-\varepsilon/2}^{t-1+\varepsilon/2} f\left( \max_{\hat{\tau}\in[\tau,\varepsilon/2]} \alpha \varphi_0^{\hat{\tau}} (s) \right) ds \\
    &\leq \frac{f'(0)+\delta_0}{\varepsilon}\int_{t-1-\varepsilon/2}^{t-1+\varepsilon/2} \max_{\hat{\tau}\in[\tau,\varepsilon/2]} \alpha \varphi_0^{\hat{\tau}} (s) ~ ds \\
    &\leq \alpha (f'(0) + \delta_0)\frac{\lambda_0}{f'(0)} \min_{\hat{\tau}\in[\tau,\varepsilon/2]} \varphi_0^{\hat{\tau}}(t),
\end{align*}
and thus \eqref{estimate1} follows. In addition, this inequality implies that
\begin{equation}\label{eq:inter-estimate}
    x_b(t) \leq \alpha\varphi_0^\tau(t), \quad \forall t\in[0,1-\tau-\varepsilon/2],
\end{equation}
due to the condition on $\delta_0$.

To prove \eqref{stepness} let $t\in(1-\tau-\varepsilon/2,1-\tau+\varepsilon/2)$. From the previous inequalities we have
\begin{align*}
    x_b'(t) &= \frac{1}{\varepsilon} [ f(x_b(t-1+\varepsilon/2)) - f(x_b(t-1-\varepsilon/2)) ] \\
    &\geq \frac{1}{\varepsilon} [f(\alpha\varphi_0^\tau(t-1+\varepsilon/2)) - f(\alpha\varphi_0^\tau(t-1-\varepsilon/2))] \\
    &\geq \alpha\frac{f'(0)+\delta_0}{\varepsilon} [\varphi_0^\tau(t-1+\varepsilon/2) - \varphi_0^\tau(t-1-\varepsilon/2)] \\
    &= \alpha (f'(0)+\delta_0)\frac{\lambda_0}{f'(0)} (\varphi_0^\tau)'(t) \\
    &\geq \alpha\pi\lambda_0\left( 1+\frac{\delta_0}{f'(0)} \right) \cos\left( \frac{\pi\varepsilon}{2} \right) \\
    &>\alpha\pi.
\end{align*}
This estimate on the derivative implies that $x_b'(t) < \alpha(\varphi_0^\eta)'(t)$ for all $\eta\in \R$ and $t$ in this interval, thus the functions $x_b$ and $\alpha\varphi_0^\eta$ intersects at most once, and therefore we deduce \eqref{stepness}.

For \eqref{estimate2}, let $t\in[1-\tau+\varepsilon/2,2-\tau-\varepsilon]$. As in the proof of \eqref{estimate1}, and using \eqref{eq:inter-estimate}, we have
\begin{align*}
    x_b(t) &= \frac{1}{\varepsilon} \int_{t-1-\varepsilon/2}^{t-1+\varepsilon/2} f(x_b(s)) ds \\
    & \geq \frac{1}{\varepsilon} \int_{t-1-\varepsilon/2}^{t-1+\varepsilon/2} f(\alpha\varphi_0^\tau(s)) ds \\
    &\geq \frac{f'(0)+\delta_0}{\varepsilon} \int_{t-1-\varepsilon/2}^{t-1+\varepsilon/2} \alpha\varphi_0^\tau(s) ds \\
    &=\alpha\lambda_0\left( 1+\frac{\delta_0}{f'(0)} \right)\varphi_0^\tau(t).
\end{align*}

Subsequently, for \eqref{estimate3} consider $t\in(2-\tau-\varepsilon,z_1(b)+1-\varepsilon/2]$, and observe that
\begin{align*}
    x_b(t) &= \frac{1}{\varepsilon} \left( \int_{t-1-\varepsilon/2}^{1-\tau-\varepsilon/2} f(x_b(s)) ds + \int_{1-\tau-\varepsilon/2}^{t-1+\varepsilon/2} f(x_b(s)) ds \right) \\
    &\geq \frac{1}{\varepsilon} \left( \int_{t-1-\varepsilon/2}^{1-\tau-\varepsilon/2} f\left( \alpha\lambda_0\left( 1+\frac{\delta_0}{f'(0)}\right)\varphi_0^\tau(s) \right) ds + \int_{1-\tau-\varepsilon/2}^{t-1+\varepsilon/2} f(\alpha\varphi_0^{1-z_1(b)}(s)) ds \right) \\
    &\geq \alpha\frac{f'(0)+\delta_0}{\varepsilon} \left( \int_{t-1-\varepsilon/2}^{1-\tau-\varepsilon/2} \lambda_0\left( 1+\frac{\delta_0}{f'(0)}\right)\varphi_0^\tau(s) ds + \int_{1-\tau-\varepsilon/2}^{t-1+\varepsilon/2} \varphi_0^{1-z_1(b)}(s) ds \right).
\end{align*}

Now, if $z_1(b)\leq 1-\tau$, it is easy to see that $\varphi_0^{\tau} \leq \varphi_0^{1-z_1(b)}$ in the first interval of integration. On the other hand, if $z_1(b)>1-\tau$ then $\lambda_0(1+\delta_0/f'(0))\varphi_0^\tau \leq \varphi_0^{1-z_1(b)}$ in the aforementioned interval. In fact, due to the shape of the functions involved, it is straightforward to check they intersect at most once in $[1/2-\tau, 1-\tau]$. Then, since
\[
    \lambda_0\left( 1+\frac{\delta_0}{f'(0)} \right) \varphi_0^\tau(1-\tau) = 0 >\varphi_0^{1-z_1(b)}(1-\tau),
\]
to get the claim it is enough to ensure that
\[
    \lambda_0\left( 1+\frac{\delta_0}{f'(0)} \right)\varphi_0^{\tau}(1-\tau-\varepsilon/2)\leq \varphi_0^{1-z_1(b)}(1-\tau-\varepsilon/2).
\]
This last inequality is equivalent to
\[
    \lambda_0\left( 1+\frac{\delta_0}{f'(0)} \right) \geq \frac{\varphi_0(1-z_1(b)-\tau-\varepsilon/2)}{\varphi_0(-\varepsilon/2)},
\]
and since $z_1(b)\in(1-\tau,1-\tau+\varepsilon/2)$, the right-hand side is bounded above by $\frac{\varphi_0(\varepsilon)}{\varphi_0(\varepsilon/2)}< 2$ and this last inequality holds. Consequently
\begin{align*}
    x_b(t) &\geq \alpha\frac{f'(0)+\delta_0}{\varepsilon} \int_{t-1-\varepsilon/2}^{t-1+\varepsilon/2} \varphi_0^{1-z_1(b)}(s) ds \\
    & = \alpha\lambda_0\left( 1+\frac{\delta_0}{f'(0)} \right) \varphi_0^{1-z_1(b)}(t),
\end{align*}
and thus \eqref{estimate3} follows.

Finally, for \eqref{stepness2} we use an analogous steepness argument as for \eqref{stepness}. Let $t\in (z_1(b)+1-\varepsilon/2, z_1(b)+1+\varepsilon/2)$, then we know that
\[
    x_b'(t) = \frac{1}{\varepsilon} [ f(x_b(t-1+\varepsilon/2)) - f(x_b(t-1-\varepsilon/2)) ].
\]
We divide the proof in two cases to use the different information about $x_b$, especially as the estimates change at $1-\tau-\varepsilon/2$. If $t-1-\varepsilon/2>1-\tau-\varepsilon/2$ then \eqref{stepness} and \eqref{estimate2} yield
\begin{align*}
    x_b'(t)& \leq \frac{1}{\varepsilon} \left[ f\left( \alpha\lambda_0\left( 1+\frac{\delta_0}{f'(0)} \right) \varphi_0^\tau(t-1+\varepsilon/2) \right) - f(\alpha\varphi_0^{1-z_1(b)}(t-1-\varepsilon/2)) \right] \\
    & \leq \alpha\frac{f'(0)+\delta_0}{\varepsilon} \left[ \lambda_0\left( 1+\frac{\delta_0}{f'(0)} \right) \varphi_0^\tau(t-1+\varepsilon/2) - \varphi_0^{1-z_1(b)}(t-1-\varepsilon/2) \right],
\end{align*}
and in the same fashion as for the previous statement, we deduce that $\lambda_0(1+\delta_0/f'(0))\varphi_0^\tau \geq \varphi_0^{1-z_1(b)}$ at $t-1+\varepsilon/2$, so that
\begin{align*}
    x_b'(t) &\leq \alpha\frac{f'(0)}{\varepsilon}[\varphi_0^{1-z_1(b)}(t-1+\varepsilon/2) - \varphi_0^{1-z_1(b)}(t-1-\varepsilon/2)] <-\alpha\pi,
\end{align*}
and we conclude just as for \eqref{stepness}. The complementary case when $t-1-\varepsilon/2\leq 1-\tau-\varepsilon/2$ is analogous.
\end{proof}

\begin{prop}[Estimates for $\tau\in(\varepsilon/2,\varepsilon{]}$] \label{prop:estimates2}
Let $b\in\B^\alpha$ with $\tau\in(\varepsilon/2,\varepsilon]$. Then, the extension $x_b$ satisfies the following estimates:
\begin{align}
    x_b(t) &\leq \lambda_0 \left( 1+\frac{\delta_0}{f'(0)} \right) \max_{\hat{\tau}\in\{\tau,\varepsilon/2\}} \alpha \varphi_0^{\hat{\tau}} (t), \quad t\in[0,1-\tau-\varepsilon/2]; \label{estimate4} \\ 
    \frac{x_b(t)}{\alpha \varphi_0^{1-z_1(b)}(t)} &>1, \quad t\in(1-\tau-\varepsilon/2,1-\tau+\varepsilon/2) \setminus \{ z_1(b) \};\label{stepness3} \\ 
    x_b(t) &\geq \alpha\lambda_0 \left( 1+\frac{\delta_0}{f'(0)} \right) \varphi_0^\tau (t), \quad t\in[1-\tau+\varepsilon/2, 2-\tau-\varepsilon];\label{estimate5} \\
    x_b(t) &\geq \alpha\lambda_0 \left( 1+\frac{\delta_0}{f'(0)} \right)\varphi_0^{1-z_1(b)}(t), \quad t\in(2-\tau-\varepsilon,z_1(b)+1-\varepsilon/2];\label{estimate6} \\
    \frac{x_b(t)}{\alpha \varphi_0^{-z_2(b)}(t)} &>1, \quad t\in(z_1(b)+1-\varepsilon/2, z_1(b)+1+\varepsilon/2)\setminus\{ z_2(b) \}. \label{stepness4}
\end{align}
\end{prop}
\begin{proof}
To prove \eqref{estimate4} let $t\in[0,1+\tau^*-\varepsilon)$. Proceeding in a similar fashion as for \eqref{estimate1} we get
\begin{align*}
    x_b(t) &= \frac{1}{\varepsilon}\int_{t-1-\varepsilon/2}^{t-1+\varepsilon/2} f(b(s))\dd s\\
    &\leq \frac{1}{\varepsilon}\int_{t-1-\varepsilon/2}^{t-1+\varepsilon/2} f(\alpha\varphi_0^{\varepsilon/2}(s))\dd s \\
    &\leq \alpha \frac{f'(0)+\delta_0}{\varepsilon} \int_{t-1-\varepsilon/2}^{t-1+\varepsilon/2} \varphi_0^{\varepsilon/2}(s) ds \\
    &= \alpha \lambda_0\left( 1+\frac{\delta_0}{f'(0)} \right) \varphi_0^{\varepsilon/2}(t).
\end{align*}
Furthermore, if $t\in[1+\tau^*-\varepsilon,1+\tau^*]$ it follows
\begin{align*}
    x_b(t) &= \frac{1}{\varepsilon} \left( \int_{t-1-\varepsilon/2}^{\tau^*-\varepsilon/2} f(b(s))\dd s + \int_{\tau^*-\varepsilon/2}^{t-1+\varepsilon/2} f(b(s)) \dd s\right) \\
    &\leq \alpha\frac{f'(0)+\delta_0}{\varepsilon} \left( \int_{t-1-\varepsilon/2}^{\tau^*-\varepsilon/2} \varphi_0^{\varepsilon/2}(s)\dd s + \int_{\tau^*-\varepsilon/2}^{t-1+\varepsilon/2} \max_{\hat{\tau}\in\{ \tau,\varepsilon/2 \}} \varphi_0^{\hat{\tau}}(s)  ~ ds \right) \\
    &\leq \alpha\frac{f'(0)+\delta_0}{\varepsilon} \int_{t-1-\varepsilon/2}^{t-1+\varepsilon/2} \varphi_0^{\varepsilon/2}(s) ds \\
    &= \alpha \lambda_0\left( 1+\frac{\delta_0}{f'(0)} \right) \varphi_0^{\varepsilon/2}(t).
\end{align*}

Similarly, we can prove that for all $t\in[1+\tau^*,1-\tau-\varepsilon/2]$ we have $x_b(t)\leq \alpha\lambda_0(1+\delta_0/f'(0))\varphi_0^{\tau}(t)$ and the result follows.

Analogously to \eqref{eq:inter-estimate}, estimate \eqref{estimate4} implies that $x_b \leq \alpha\varphi_0^\tau$ on $[0,1-\tau-\varepsilon/2]$. Indeed, by the condition on $\delta_0$ is clear that $\lambda_0(1+\delta_0/f'(0)) \varphi_0^\tau \leq \alpha \varphi_0^\tau$ on $[0,1-\tau-\varepsilon/2]$; while, since $\lambda_0(1+\delta_0/f'(0)) \varphi_0^{\varepsilon/2}$ and $\alpha\varphi_0^\tau$ intersects at most once in $[-\varepsilon/2, 1/2 - \varepsilon/2]$, it is enough to check that
\[
    \lambda_0\left( 1+\frac{\delta_0}{f'(0)} \right) \varphi_0^{\varepsilon/2}(0) \leq \varphi_0^\tau(0).
\]
But this is equivalent to
\[
    \lambda_0\left( 1+\frac{\delta_0}{f'(0)} \right) \geq \frac{\varphi_0(1+\tau)}{\varphi_0(1+\varepsilon/2)}= \frac{\varphi_0(\tau)}{\varphi_0(\varepsilon/2)},
\]
and this holds since the right-hand side is bounded above by 2.

The proof of \eqref{stepness3} follows the same steps as the one for \eqref{stepness}.

Next, for \eqref{estimate5} let $t\in[1-\tau+\varepsilon/2,2+\tau^*-\varepsilon/2]$, then
\begin{align*}
    x_b(t) &= \frac{1}{\varepsilon}\left( \int_{t-1-\varepsilon/2}^{0} f(b(s))\dd s + \int_0^{t-1+\varepsilon/2} f(x_b(s))\right) \\
    &\geq \alpha\frac{f'(0)+\delta_0}{\varepsilon}\left( \int_{t-1-\varepsilon/2}^0 \varphi_0^{\tau}(s)\dd s + \int_0^{t-1+\varepsilon/2} \lambda_0\left( 1+\frac{\delta_0}{f'(0)} \right) \varphi_0^{\varepsilon/2}(s) \dd s \right) \\
    &\geq \alpha\frac{f'(0)+\delta_0}{\varepsilon} \int_{t-1-\varepsilon/2}^{t-1+\varepsilon/2} \varphi_0^\tau(s) ds \\
    &= \alpha\lambda_0\left( 1+\frac{\delta_0}{f'(0)} \right) \varphi_0^\tau(t).
\end{align*}
On the other hand, if $t\in[2+\tau^*-\varepsilon/2,2+\tau^*+\varepsilon/2]$ we obtain
\begin{align*}
    x_b(t) &= \frac{1}{\varepsilon}\left( \int_{t-1-\varepsilon/2}^{1+\tau^*} f(x_b(s))\dd s + \int_{1+\tau^*}^{t-1+\varepsilon/2} f(x_b(s))\dd s \right) \\
    &\geq \alpha \frac{f'(0)+\delta_0}{\varepsilon}\left( \int_{t-1-\varepsilon/2}^{1+\tau^*} \lambda_0 \left( 1+\frac{\delta_0}{f'(0)} \right) \varphi_0^{\varepsilon/2}(s)\dd s + \int_{1+\tau^*}^{t-1+\varepsilon/2} \lambda_0 \left( 1+\frac{\delta_0}{f'(0)} \right) \varphi_0^{\tau}(s)\dd s \right) \\
    &\geq \alpha\lambda_0 \left( 1+\frac{\delta_0}{f'(0)} \right) \varphi_0^{\tau}(t).
\end{align*}

In a similar fashion we can prove that this inequality holds on $[2+\tau^*+\varepsilon/2,2-\tau-\varepsilon]$ and we obtain the result

    The proofs of \eqref{estimate6} and \eqref{stepness4} follow from similar arguments as the ones in the proof \eqref{estimate3} and \eqref{stepness2}, and we omit them for concision. Proposition \ref{prop:estimates2} is proved.
\end{proof}

\subsection{Connection with a convex set}

The notion of the set $\B^\alpha$ comes from trying to identify a subsolution to the problem that allows us to construct an invariant set in which apply a fixed point argument. However, due to the flexibility allowed to the parameter $\tau$, this set fails to be convex. Thus, we are dealing with two different issues. It is necessary to take $\tau$ in this way taking into account the lack of symmetry in the problem, because this may raise solutions with an a priori unknown half-period. On the other hand, one of the usual conditions in the fixed point theorems is the convexity of the set we are working on, a property which $\B^\alpha$ does not possess.

One possibility to handle these problems is to use a topological degree argument, since this does not require to have a convex set, but instead consider an open set (see \cite{zeidler}). This additional requirement requires providing an stronger topology to the state space, as the $C^1$-topology. This will also carry several difficulties to define and prove the main properties of the Poincaré map.

One other way in which we can deal with this problem is to transform the set in one which turns out to be convex. Or, more generally, one which possesses the fixed point property. To do this, note that we can divide a function $b\in \B^\alpha$ in two parts, its positive part defined in $[-1-\varepsilon/2,-\tau]$, and its negative part defined in $[-\tau,0]$. Thus, by a simple change of variables we can regard these functions as defined on $[-1-\varepsilon/2,-\varepsilon/2]$ and $[-\varepsilon/2,0]$, respectively. That is, we transform the set $\B^\alpha$ into one which does not include explicitly the value of $\tau$, and therefore in a convex set, and this is the approach we are going to take.

To formalize the aforementioned ideas, as when we define $\B^\alpha$, let us establish the following assumptions.
\begin{assumption}\label{as:Ualpha}
Let $b_1\in C_0([-1-\varepsilon/2,-\varepsilon/2])$ and $b_2 \in C_0([-\varepsilon/2,0])$. The following conditions are satisfied:
\begin{enumerate}[label={\roman*)}]
    \item $b_1(-\varepsilon/2)=0$.
    \item There exists a unique $\tau\in(0,\varepsilon]$ such that the following estimates hold
\begin{align*}
    b_1(t) &\geq \gamma_\tau(h_\tau(t)), \quad \forall t\in[-1-\varepsilon/2,-\varepsilon/2], \\
    b_2(t) &\leq \alpha \varphi_0^\tau \left( \frac{2\tau}{\varepsilon}t \right), \quad \forall t\in[-\varepsilon/2,0],
\end{align*}
where $h_\tau \colon [-1-\varepsilon/2, -\varepsilon/2]\to [-1-\varepsilon/2,-\tau]$ is given by $h_\tau(t)=t-(\tau-\varepsilon/2)(t+1+\varepsilon/2)$.
\end{enumerate}
\end{assumption}

\begin{assumption}\label{as:Valpha}
Let $b_1\in C_0([-1-\varepsilon/2, -\varepsilon/2])$ and $b_2\in C_0([-\varepsilon/2,0])$. The following conditions are satisfied:
\begin{enumerate}[label={\roman*)}]
    \item $b_1\geq 0$ and $b_2\leq 0$.
    \item $b_1(-\varepsilon/2) =0$.
\end{enumerate}
\end{assumption}

Then, let $E=C_0([-1-\varepsilon/2,-\varepsilon/2])\times C_0([-\varepsilon/2,0])\times (0,\varepsilon]$, and we define the following sets for $\tau_0\in(0,\varepsilon/2)$ fixed
\begin{align}
    B^{\alpha,\tau_0} &= \{ b\in C_0([-1-\varepsilon/2,0]) \colon \text{Assumption \ref{as:Balpha} holds with}~\tau\geq \tau_0 \},\\
    U^{\alpha,\tau_0} &= \{ (b_1,b_2,\tau)\in E \colon \text{Assumption \ref{as:Ualpha} holds with}~\tau\geq \tau_0 \},\\
    V^{\alpha,\tau_0} &= \{ (b_1,b_2,\tau)\in E \colon \text{Assumption \ref{as:Valpha} holds and}~\tau\in [\tau_0, \varepsilon]\}.
\end{align}

Now, our aim is to prove that $B^{\alpha,\tau_0}$ is homeomorphic to $V^{\alpha,\tau_0}$, and to this end we define the maps $B^{\alpha,\tau_0} \xrightarrow{\xi_1} U^{\alpha,\tau_0} \xrightarrow{\xi_2} V^{\alpha,\tau_0}$ given by
\[
    \xi_1(u) = (v_1,v_2,\tau), \quad \text{and}\quad \xi_2(v_1,v_2,\tau) = (w_1,w_2,\tau),
\]
with
\begin{align*}
    v_1(t)&=u(h_\tau(t)), \quad t\in[-1-\varepsilon/2,-\varepsilon/2],\\
    v_2(t) &= u\left( \frac{2\tau}{\varepsilon} t \right), \quad t\in[-\varepsilon/2,0], \\
    w_1(t) &= v_1(t) -\gamma_\tau(h_\tau(t)), \quad t\in[-1-\varepsilon/2,-\varepsilon/2],\\
    w_2(t)&= v_2(t)- \alpha\varphi_0^{\tau}\left( \frac{2\tau}{\varepsilon}t \right), \quad t\in[-\varepsilon/2,0],
\end{align*}
where $h_\tau$ is as in Assumption \ref{as:Ualpha}.

\begin{prop}
    The maps $\xi_1$ and $\xi_2$ are well defined.
\end{prop}
\begin{proof}
Let $u\in B^{\alpha,\tau_0}$ and let $\tau\in [\tau_0,\varepsilon]$ be the unique zero of $u$. Since $u(t)\geq \gamma_\tau(t)$ for $t\in[-1-\varepsilon,-\tau]$, and $u(t)\leq \alpha \varphi_0^{\tau}(t)$ for $t\in[-\tau,0]$, then
\begin{align*}
    v_1(t) &\geq \gamma_\tau(h_\tau(t)), \quad \forall t\in[-1-\varepsilon/2, -\varepsilon/2], \\
    v_2(t) & \leq \alpha\varphi_0^{\tau} \left(\frac{2\tau}{\varepsilon} t\right), \quad \forall t\in[-\varepsilon/2,0],
\end{align*}
where $\xi_1(u)=(v_1,v_2,\tau)$. Moreover, $v_1(-\varepsilon/2)=v_2(-\varepsilon/2)=u(-\tau)=0$ and consequently, $\xi_1(u)=(v_1,v_2,\tau)\in U^{\alpha,\tau_0}$.

It is straightforward to check that $\xi_2$ is well defined.
\end{proof}

\begin{prop}
    The maps $\xi_1$ and $\xi_2$ are continuous.
\end{prop}
\begin{proof}
Let $u\in B^{\alpha,\tau_0}$ and $(u_n)_n\subseteq B^{\alpha,\tau_0}$ such that $u_n \to u$. In addition, let $\xi_1(u_n)=(v_{1,n},v_{2,n},\tau_n)$ and $\xi_1(u)=(v_1,v_2,\tau)$.

First, let us prove $\tau_n \to \tau$. To do this, we first check that $(\tau_n)_n$ indeed converges. Let $(\tau_{m_k})_k$ and $(\tau_{n_k})_k$ two convergent subsequences and assume they converge to two different limits, say $\tau_0$ and $\tau_1$. Without loss of generality, we assume that $\tau_0<\tau_1$ and therefore, for a fixed $\delta>0$ and taking $k$ sufficiently large such that $-\tau_{m_k}>-\tau_0 - \delta$, then
\[
    u_{m_k}(t) \geq \gamma_{\tau_{m_k}}(t), \quad \forall t\in[-1-\varepsilon/2, -\tau_0-\delta].
\]
Noting that $\tau_{m_k}\to \tau_0$ implies $\gamma_{\tau_{m_k}}\to \gamma_{\tau_0}$ uniformly in this set, then passing to the limit as $k\to \infty$ yields
\[
    u(t) \geq \gamma_{\tau_0}(t), \quad \forall t\in[-1-\varepsilon/2,-\tau_0-\delta],
\]
and since $\delta>0$ is arbitrary it follows
\begin{equation}\label{eq:21}
    u(t) \geq \gamma_{\tau_0}(t), \quad \forall t\in[-1-\varepsilon/2,-\tau_0].
\end{equation}
In a complementary fashion, making the same manipulations with the other subsequence, we obtain
\begin{equation}
    u(t) \leq \alpha \varphi_0^{\tau_1}(t) \quad \forall t\in[-\tau_1,0],
\end{equation}
but this contradicts \eqref{eq:21} since it states that $u$ is positive on $(-\tau_1,-\tau_0)$, while the second asserts $u$ is negative in the same interval. Thus, we deduce that $\tau_0=\tau_1$, and then the sequence $(\tau_n)_n$ is convergent.

Now, we assume that $\tau_n\to \tilde{\tau}$, and $\tilde{\tau}\neq \tau$; without loss of generality we take $\tilde{\tau}<\tau$. Then, as before we have
\[
    u(t) \geq \gamma_{\tilde{\tau}}(t) \quad \forall t\in[-1-\varepsilon/2,-\tilde{\tau}],
\]
owing the definition $\tau$ it follows
\[
    u(t)\leq \alpha \varphi_0^\tau (t) \quad \forall t\in[-\tau,0],
\]
and this is also a contradiction. Consequently, $\tilde{\tau}=\tau$.

Next, observe that
\begin{align*}
    |v_1(t) - v_{1,n}(t)| &= |u(h_\tau(t)) - u_n(h_{\tau_n}(t))| \\
    &\leq |u(h_\tau(t)) - u(h_{\tau_n}(t))| + |u(h_{\tau_n}(t)) - u_n(h_{\tau_n}(t))|
\end{align*}
and clearly the right-hand side converges uniformly to 0 since $\tau_n\to \tau$ implies the uniform convergence of $h_{\tau_n}\to h_\tau$, and also thanks to the uniform convergence $u_n\to u$. In the same fashion, we can prove that $v_{2,n}\to v_2$ uniformly, and for that reason $\xi_1$ is a continuous function.

The continuity of $\xi_2$ is straightforward since the maps $\tau\mapsto \gamma_\tau(h_\tau)$ and $\tau\mapsto \varphi_0^\tau\left( \frac{2\tau}{\varepsilon} \cdot \right)$ are continuous.
\end{proof}

It is easy to see that the functions $\xi_1$ and $\xi_2$ are one-to-one, and therefore invertible on the sets $\xi_1(B^{\alpha,\tau_0})$ and $\xi_2(U^{\alpha,\tau_0})$, which actually turn out to be $U^{\alpha,\tau_0}$ and $V^{\alpha,\tau_0}$, respectively. Moreover, the inverses are given by
\begin{align*}
    \xi_1^{-1}(v_1,v_2,\tau) = u, \quad\text{and}\quad \xi_2^{-1}(w_1,w_2,\tau) = (v_1,v_2,\tau),
\end{align*}
with
\begin{align*}
    u(t)&= \begin{cases}
        v_1(h_\tau^{-1}(t)), & t\in[-1-\varepsilon/2,-\tau], \\
        v_2\left( \frac{\varepsilon}{2\tau} t \right), & t\in[-\tau,0],
    \end{cases}\\
    v_1(t)&= w_1(t) + \gamma_\tau(h_\tau(t)), \quad t\in[-1-\varepsilon/2,-\varepsilon/2], \\
    v_2(t) &= w_2(t) + \alpha\varphi_0^{\tau} \left( \frac{2\tau}{\varepsilon} t \right), \quad t\in[-\varepsilon/2,0],
\end{align*}
where $h_\tau^{-1} \colon [-1-\varepsilon/2, -\tau] \to [-1-\varepsilon/2,-\varepsilon/2]$ is given by $h_\tau^{-1}(t) = \frac{t + (\tau-\varepsilon/2)(1+\varepsilon/2)}{1-\tau + \varepsilon/2}$, i.e., $h_\tau^{-1}$ is the inverse function of $h_\tau$.

\begin{prop}
    The maps $\xi_1^{-1}$ and $\xi_2^{-1}$ are continuous.
\end{prop}
\begin{proof}
Let $(v_1,v_2,\tau)\in U^{\alpha,\tau_0}$ and $(v_{1,n},v_{2,n},\tau_n)_n \subseteq U^{\alpha,\tau_0}$ which converges to $(v_1,v_2,\tau)$. Also, let $u_n=\xi_n^{-1}(v_{1,n},v_{2,n},\tau_n)$ and $u=\xi_1^{-1}(v_1,v_2,\tau)$. We consider $I= \{ n\in\N \colon \tau\geq \tau_n \}$, and the following analysis will be made for the elements of the sequence with index in $I$; for the elements with index in $\N\setminus I$ it is done in a similar fashion.

Note that, if $t\in[-1-\varepsilon/2,-\tau]$, we have
\begin{align*}
    |u_n(t) - u(t)| &= |v_{1,n}(h_{\tau_n}^{-1}(t)) - v_1(h_\tau^{-1}(t))| \\
    &\leq \|v_{1,n}-v_1\| + | v_1(h_{\tau_n}^{-1}(t)) - v_1(h_\tau^{-1}(t)) |;
\end{align*}
on the other hand, if $t\in[-\tau_n,0]$ then
\begin{align*}
    |u_n(t) - u(t)| &= \left| v_{2,n}\left( \frac{\varepsilon}{2\tau_n} t \right) - v_2\left( \frac{\varepsilon}{2\tau} t\right) \right| \\
    &\leq \| v_{2,n} - v_2 \| + \left| v_2\left( \frac{\varepsilon}{2\tau_n}t \right) - v_2\left( \frac{\varepsilon}{2\tau} t \right) \right|,
\end{align*}
and from these estimates it is straightforward to check that taking the maximum over the interval under consideration and subsequently the limit as $n\to \infty$ ($n\in I$) goes to 0. Furthermore, if $t\in[-\tau,-\tau_n]$ it holds
\begin{align*}
    |u_n(t)- u(t)| & = \left| v_{1,n}(h_{\tau_n}^{-1}(t)) - v_2\left( \frac{\varepsilon}{2\tau} t \right) \right|,
\end{align*}
and since all the functions involved are continuous there exists $t_n\in[-\tau,-\tau_n]$ such that
\[
    \max_{t\in[-\tau,-\tau_n]} |u_n(t) - u(t)| = \left|v_1(h_{\tau_n}^{-1}(t_n)) - v_2\left( \frac{\varepsilon}{2\tau} t_n \right) \right|,
\]
from where
\[
    \limsup_{\substack{n\to +\infty \\ n\in I}}\max_{t\in[-\tau,-\tau_n]} |u_n(t) - u(t)| = \left| v_1(h_{\tau}^{-1}(-\tau) - v_2(\varepsilon/2)) \right| = 0,
\]
owing to the assumptions about $v_1$ and $v_2$.

Thus, observing that
\begin{align*}
    \max_{t\in[-1-\varepsilon/2,0]} |u_n(t)-u(t)| = \max & \left\{ \max_{t\in[-1-\varepsilon/2,-\tau]} |u_n(t)-u(t)| , \max_{t\in[-\tau,-\tau_n]} |u_n(t)-u(t)|, \right. \\
    & \left. \max_{t\in[-\tau_n,-\tau]} |u_n(t)-u(t)| \right\},
\end{align*}
we can conclude that
\[
    \lim_{\substack{n\to +\infty \\ n\in I}}\max_{t\in[-1-\varepsilon/2,0]} |u_n(t) - u(t)| = 0.
\]
This shows that $u_n \to u$ uniformly and therefore $\xi_1^{-1}$ is continuous.

The proof of the continuity for $\xi_2^{-1}$ is the same as for $\xi_2$.
\end{proof}

Therefore, we have proved that $B^{\alpha,\tau_0}$ is homeomorphic to $V^{\alpha,\tau_0}$.

\subsection{Invariance of the state space}

For a given $R>0$, let
\begin{align*}
    B^{\alpha,\tau_0}_R &= \{ b\in B^{\alpha,\tau_0} \colon \| b \|\leq R \}, \\
    U^{\alpha,\tau_0}_R &= \{ (b_1,b_2,\tau)\in U^{\alpha,\tau_0} \colon \|b_1\|, \|b_2\|\leq R \}, \\
    V^{\alpha,\tau_0}_R &= \{ (b_1,b_2,\tau)\in V^{\alpha,\tau_0} \colon \|b_1\|,\|b_2\| \leq R \}.
\end{align*}

The main motivation to define the sets $B^{\alpha,\tau_0}_R$ is given in the next lemma

\begin{lem}\label{lem:invariance_B}
    Let $R\geq \max_{|x|\leq A_0} |f(x)|$. Then, the set $B^{\alpha,\tau_0}_R$ is forward invariant under the operator $\Fcal$, defined in \eqref{eq:poincare_operator},  for every $\tau_0$ such that
    \begin{equation}\label{eq:tau0}
        \tau_0 \leq \frac{\alpha\varepsilon \lambda_0 \varphi_0(\varepsilon/2)}{2R} \left( 1+\frac{\delta_0}{f'(0)} \right),
    \end{equation}
    where $\varepsilon,\delta_0$ and $\alpha$ are like at the beginning of this section.
\end{lem}
\begin{proof}
Let $b\in B^{\alpha,\tau_0}_R$ and $u=\Fcal(b)$. Observe that $\|u\|\leq R$ owing to Proposition \ref{unif_bound}. On the other hand, since the estimates $x_b$ verifies are the same in the two different cases for the associated $\tau$ (see Propositions \ref{prop:estimates1} and \ref{prop:estimates2}), without loss of generality let $\tau\in[\tau_0,\varepsilon/2]$ be the associated zero of $b$ given by the definition of $B^{\alpha,\tau_0}_R$. If $z_2(b)$ is the second positive zero of $x_b$, is straightforward to check $z_2(b)\in(z_1(b)+1-\varepsilon/2, z_1(b)+1+\varepsilon/2  )$. Then, note that $\tau_u := z_1(b)+1+\varepsilon/2 - z_2(b)$ verifies $u(-\tau_u)=0$. Thanks to the Proposition \ref{prop:estimates1}, in order to ensure that $u\in \B^\alpha$, we just need to prove that
\[
    u(t) \geq \gamma_{\tau_u}(t), \quad \forall t\in[ -\tau-z_1(b), -\varepsilon ],
\]
and, indeed, owing the estimates established there, it is enough to show
\begin{align*}
    \alpha\lambda_0 \left( 1+\frac{\delta_0}{f'(0)} \right) \varphi_0^\tau(z_1(b)+1+\varepsilon/2+t) &\geq \gamma_{\tau_u}(t) \quad \forall t\in[-\tau-z_1(b),1-\tau-z_1(b)-\frac{3}{2}\varepsilon]; \\
    \alpha\lambda_0\left( 1+\frac{\delta_0}{f'(0)} \right) \varphi_0^{1-z_1(b)}(z_1(b)+1+\varepsilon/2+t)&\geq \gamma_{\tau_u}(t) \quad \forall t\in[1-\tau-z_1(b)-\frac{3}{2}\varepsilon,-\varepsilon].
\end{align*}

Due to the shape of the functions involved, it can be seen that they intersect at most once in the intervals $[-\tau-z_1(b), 1-\tau - z_1(b) - (3/2)\varepsilon]$ and $[1-\tau - z_1(b) - (3/2)\varepsilon , -\varepsilon]$. Therefore, these estimates follow if they hold at the boundary points of the intervals; i.e., if
\begin{align*}
    \alpha\lambda_0 \left( 1+\frac{\delta_0}{f'(0)} \right) \varphi_0^\tau(1-\tau+\varepsilon/2) &\geq \alpha\varphi_0^{\varepsilon/2}(-\tau-z_1(b)) \\
    \alpha\lambda_0 \left( 1+\frac{\delta_0}{f'(0)} \right) \varphi_0^\tau(2-\tau-\varepsilon) &\geq \alpha\varphi_0^{\tau_u}(1-\tau-z_1(b)-\frac{3}{2}\varepsilon) \\
    \alpha\lambda_0 \left( 1+\frac{\delta_0}{f'(0)} \right) \varphi_0^{1-z_1(b)}(z_1(b)+1-\varepsilon/2) &\geq \alpha\varphi_0^{\tau_u}(-\varepsilon).
\end{align*}

The first inequality is equivalent to
\[
    \lambda_0\left( 1+\frac{\delta_0}{f'(0)} \right) \geq \frac{\varphi_0(1-\tau-z_1(b) + \varepsilon/2)}{ \varphi_0(\varepsilon/2)},
\]
and since $z_1(b)\in (1-\tau-\varepsilon/2, 1-\tau+\varepsilon/2)$ we get that the left-hand side is bounded above by $\varphi_0(\varepsilon)/\varphi_0(\varepsilon/2)<2$, and that inequality holds. The second one is equivalent to
\begin{align*}
    \lambda_0\left( 1+\frac{\delta_0}{f'(0)} \right) &\geq \frac{\varphi_0(2+\tau_u-\tau-z_1(b)-\frac{3}{2}\varepsilon)}{\varphi_0(3-\varepsilon)} \\
    &= \frac{\varphi_0(1-\tau-z_2(b)-\varepsilon)}{\varphi_0(1-\varepsilon)} \\
    &= \frac{\varphi_0(\tau+z_2(b)+\varepsilon)}{\varphi_0(\varepsilon)},
\end{align*}
and due to the bounds on $z_1(b)$ and $z_2(b)$ we deduce that the second term is bounded above by $\varphi_0(2\varepsilon)/\varphi(\varepsilon)<2$, which holds. Finally, the last inequality is equivalent to
\begin{align*}
    \lambda_0\left( 1+\frac{\delta_0}{f'(0)} \right) &\geq \frac{\varphi_0(1+\tau_u-\varepsilon)}{\varphi_0(3-\varepsilon/2)} \\
    &= \frac{\varphi_0(2+z_1(b)-z_2(b)-\varepsilon/2)}{\varphi_0(1-\varepsilon/2)} \\
    &= \frac{\varphi_0(z_1(b)-z_2(b)-\varepsilon/2)}{\varphi_0(\varepsilon/2)},
\end{align*}
as before, the left-hand side is bounded above by 2 an the result follows. 

To conclude the proof, we need to check that $\tau_u\geq \tau_0$. For this, observe
\begin{align*}
    u(0) &= x_b(z_1(b)+1+\varepsilon/2) \\
    &= \frac{1}{\varepsilon} \left( \int_{z_1(b)}^{1-\tau+\varepsilon/2} f(x_b(s)) ds + \int_{1-\tau+\varepsilon/2}^{z_1(b)+\varepsilon} f(x_b(s)) ds \right) \\
    &\leq \frac{1}{\varepsilon} \left(  \int_{z_1(b)}^{1-\tau+\varepsilon/2} f(\alpha\varphi_0^{1-z_1(b)}(s)) ds + \int_{1-\tau+\varepsilon/2}^{z_1(b)+\varepsilon} f\left( \alpha\lambda_0\left( 1+\frac{\delta_0}{f'(0)} \right) \varphi_0^\tau(s) \right) ds \right) \\
    &\leq \alpha\frac{f'(0)+\delta_0}{\varepsilon} \left(  \int_{z_1(b)}^{1-\tau+\varepsilon/2} \varphi_0^{1-z_1(b)}(s) ds + \int_{1-\tau+\varepsilon/2}^{z_1(b)+\varepsilon} \lambda_0 \left( 1+\frac{\delta_0}{f'(0)} \right) \varphi_0^\tau(s) ds \right) \\
    &\leq \alpha\frac{f'(0)+\delta_0}{\varepsilon} \int_{z_1(b)}^{z_1(b)+\varepsilon} \varphi_0^{1-z_1(b)}(s) ds \\
    &= \alpha\lambda_0\left( 1+\frac{\delta_0}{f'(0)} \right) \varphi_0^{1-z_1(b)}(z_1(b)+1+\varepsilon/2) \\
    &= -\alpha\lambda_0\left( 1+\frac{\delta_0}{f'(0)} \right) \varphi_0\left( \frac{\varepsilon}{2} \right).
\end{align*}
Also, observe that the derivative of $u$ is bounded almost everywhere from above by $\frac{2}{\varepsilon}\max_{|x|\leq R} |f(x)|\leq \frac{2R}{\varepsilon}$, and therefore, by the Mean Value Theorem,
\[
    u(t) \leq u(0) - \frac{2R}{\varepsilon}t \leq - \alpha\lambda_0 \left( 1+\frac{\delta_0}{f'(0)} \right) \varphi_0\left( \frac{\varepsilon}{2} \right) - \frac{2R}{\varepsilon}t, \quad t\in[-1-\varepsilon/2,0],
\]
which shows that $\tau_u \geq \tau_0$. Therefore, we have proved that $B^{\alpha,\tau_0}_R$ is a forward invariant set.
\end{proof}

Note that all the sets $V^{\alpha,\tau_0}_R$ possess the fixed-point-property since they are closed and convex subsets of a Banach space (owing to the Schauder fixed point theorem). Therefore, as this property is preserved by homeomorphisms, we are interested in finding a set of the form $\xi_1^{-1}(\xi_2^{-1}(V^{\alpha,\tau_0}_R))\subseteq B^{\alpha,\tau_0}$ which is forward invariant under $\Fcal$.

\begin{prop}\label{prop:invariance}
Let $R\geq \max_{|x|\leq A_0}|f(x)|$. Then, for every $\alpha>0$ small enough the set $\xi_1^{-1}(\xi_2^{-1}(V^{\alpha,\tau_0}_R))$ is forward invariant under $\Fcal$ for every $\tau_0$ that verifies \eqref{eq:tau0}.
\end{prop}
\begin{proof}
First of all, note that
\begin{align*}
    \xi_2^{-1}(V^{\alpha,\tau_0}_R) \subseteq U^{\alpha,\tau_0}_{R+\alpha},\quad
    \xi_1^{-1}(U^{\alpha,\tau_0}_R) \subseteq B^{\alpha,\tau_0}_R.
\end{align*}
Conversely
\begin{align*}
    \xi_1(B^{\alpha,\tau_0}_R) \subseteq U^{\alpha,\tau_0}_R, \quad
    \xi_2(U^{\alpha,\tau_0}_R)\subseteq V^{\alpha,\tau_0}_R.
\end{align*}

Now, let $u\in \xi_1^{-1}(\xi_2^{-1}(V^{\alpha,\tau_0}_R))$, we need to check that $\Fcal(u)\in \xi_1^{-1}(\xi_2^{-1}(V^{\alpha,\tau_0}_R))$. From the previous relations, it is enough to check that $\Fcal(u)\in B^{\alpha,\tau_0}_R$. As $u\in B^{\alpha,\tau_0}_{R+\alpha}$, Lemma \ref{lem:invariance_B} implies $\Fcal(u)$ verifies Assumption \ref{as:Balpha} with $\tau \geq \tau_0$, so we just need to check $\|\Fcal(u)\|\leq R$. To this, observe that
\[
    \Fcal(u)(t) = \frac{1}{\varepsilon}\int_{z_1(u)+t}^{z_1(u)+t+\varepsilon} f(x_u(s))\dd s,
\]
from where, by Proposition \ref{unif_bound}, for every $t\in[-1-\varepsilon/2,0]$
\begin{align*}
    |\Fcal(u)(t)| &\leq \max_{|x|\leq R+\alpha} |f(x)|.
\end{align*}

If $R<A_0$ then for each $\alpha$ small enough such that $R + \alpha\leq A_0$ we get $|\Fcal(u)(t)|\leq R$ for $t\in[-1-\varepsilon/2,0]$. Furthermore, if $R\geq A_0$ then using \eqref{diag_cond} yields
\[
    \max_{A_0\leq |x|\leq R} |f(x)| < R,
\]
and therefore the continuity of $f$ assures us that for every $\alpha$ small enough it holds
\[
    \max_{A_0\leq |x|\leq R+\alpha} |f(x)| \leq R,
\]
from where
\[
    \max_{|x|\leq R+ \alpha} |f(x)| = \max\left\{ \max_{|x|\leq A_0} |f(x)| , \max_{A_0\leq|x|\leq R+\alpha} |f(x)| \right\} \leq R,
\]
and therefore it follows $|\Fcal(u)(t)|\leq R$ for every $t\in[-1-\varepsilon/2,0]$. This completes the proof.
\end{proof}

\section{Proofs of the main results}\label{sec:proofs}

\subsection{Proof of Theorem \ref{thm:slowly-oscillating-extension}}

Thanks to the Proposition \ref{deriv_x} we can apply the same reasoning used in Remark \ref{rem:boundontau} to deduce that $z_j(b)\in (z_{j-1}(b) + 1 -\varepsilon/2, z_{j-1}(b) + 1 + \varepsilon/2)$ for $j\geq1$, where $z_0(b) := -\tau$; and that $(-1)^{j+1} x_b(t)>0$ for $t\in(z_j(b) , z_{j+1}(b))$. Therefore, two consecutive zeros of the extension $x_b$ are at least distanced by $1-\varepsilon/2$ and it follows that $x_b$ is a slowly oscillating solution of the equation \ref{eq:prob}.\qed

\subsection{Proof of Theorem \ref{thm:existence-solutions}}

Due to Proposition \ref{prop:invariance} we know that for every suitable value $R>0$ the set $\xi_1^{-1}(\xi_2^{-1}(V^{\alpha, \tau_0}_R)) \subseteq B^\alpha$ is invariant under the Poincaré map $\Fcal$. Note that it is necessary, in the first instance, to set a value of $\varepsilon$ such that $\lambda_0 = -f'(0) \frac{2}{\pi\varepsilon} \sin\left( \frac{\pi\varepsilon}{2} \right) >2$; and thus the limiting value $\varepsilon_0$ comes from solving the equation
\[
    -f'(0)\frac{2}{\pi\varepsilon_0} \sin\left( \frac{\pi\varepsilon_0}{2} \right) = 2.
\]
Finally, the existence of the periodic solution comes from the fact that $V^{\alpha,\tau_0}_R$ possesses the fixed-point property: indeed it is a convex set in a Banach space, therefore any compact map on $V^{\alpha,\tau_0}_R$ has a fixed point. \qed

\begin{remark}
It is possible to give sharper estimates on the norm of periodic solutions given by Theorem \ref{thm:existence-solutions}. In fact, noting that the main property we use for the value $A_0$ is given in \eqref{diag_cond}, and to maintain this condition it is enough to have $A_0$ as any number strictly larger than $A := \max\{ |\kappa_1|,\kappa_2 \}$. As a consequence, by means of a limit argument, we deduce the existence of a periodic solution in the set $\xi_1^{-1}(\xi_2^{-1}(V^{\alpha,\tau_0}_R))$, with $R = \max_{|x|\leq A} |f(x)|$.
\end{remark}

\subsection{Proof of Theorem \ref{thm:limit-behavior}}

In this paragraph we are interested in the behavior as $\varepsilon$ goes to 0 in a particular symmetric case. Note that taking $\varepsilon\to 0$ in \eqref{eq:prob} formally yields the following difference equation with continuous nonlinearity $f$: 
\[
    b(t) = f(b(t-1)), \quad t\geq 0.
\]
The dynamics of the solutions to this equation relies completely in the geometric properties of $f$, and its asymptotic behavior can be classified in three different classes: asymptotically constant, relaxation type, and turbulent type solutions; following the classification proposed in \cite{sharkovsky2}. In particular, for the class of functions we are considering it can be proved that all the solutions converge, in a suitable topology, to a square-wave function determined by the points of period two of the function $f$. We refer to  \cite{sharkovsky1} for a complete study and discussion on this equation.

Results concerning to the convergence towards this type of solutions, i.e. of square-wave type, have been founded for a singular perturbation of the previous equation, mainly to
\[
    \varepsilon b'(t) = -b(t) + f(b(t-1)), \quad t\geq 0,
\]
see for example \cite{sharkovsky1,nussbaum_mallet}. Also, results of this type have been found for the case of an odd nonlinearity in the integral equation we considered in this work, as can be seen in \cite{CDM}. Motivated by this, in this section we state a convergence result in a weak-symmetric case in which $f$ is not necessarily an odd function, see Figure \ref{fig:4}.

Let $a\in (f'(0),-1)$ and $\kappa_a^+,\kappa_a^-$ be the intersection points between $f$ and the diagonal with slope $a$. Owing to the conditions on the derivatives at $\pm \kappa_0$, taking $a$ close enough to $-1$, there  holds
\begin{equation}\label{eq:diag_condition2}
    \frac{f(x)}{x} < a, \quad x \in(\kappa_a^-, \kappa_a^+) \setminus\{0\}.
\end{equation}

With this in mind, solving the eigenvalue problem
\begin{equation}\label{eq:linear_a}
    \lambda \varphi(t) = \frac{a}{\varepsilon} \int_{t-1-\varepsilon/2}^{t-1+\varepsilon/2} \varphi(s) \dd s
\end{equation}
yields that $\varphi_k(t) = \sin(k \pi t)$ is an eigenfunction for every $k\in\Z$ associated to the eigenvalue $\lambda_{a,k} = -\frac{2a}{k\pi\varepsilon} \sin\left( \frac{k \pi \varepsilon}{2} \right)$. Observe that $\lambda_{a,k}\to -a$ as $\varepsilon\to 0^+$, so we can assume that $\lambda_{a,k}>1$. Thus, if $\varphi_{k,+} := \max(\varphi_k,0)$ then
\[
    \lambda_{a,k} \varphi_{k,+}(t) \geq \lambda_{a,k} \varphi_k(t) = \frac{a}{\varepsilon} \int_{t-1-\varepsilon/2}^{t-1+\varepsilon/2} \varphi_k(s) \dd s \geq \frac{a}{\varepsilon} \int_{t-1-\varepsilon/2}^{t-1+\varepsilon/2} \varphi_{k,+}(s) \dd s,
\]
from where $\varphi_{k,+}$ is a sub-solution of \eqref{eq:linear_a}. From \eqref{eq:linear_a} it also follows that
\begin{equation}\label{eq:lambda_ak}
    \lambda_{a,k} \varphi_k(t) = (-1)^k\frac{a}{\varepsilon} \int_{t-\varepsilon/2}^{t+\varepsilon/2} \varphi_k(s) \dd s,
\end{equation}
and therefore, by taking $k$ an odd number we obtain
\begin{equation}\label{eq:subsolution}
    \lambda_{a,k} \varphi_k(t) = -\frac{a}{\varepsilon} \int_{t-\varepsilon/2}^{t+\varepsilon/2} \varphi_k(s) \dd s \leq -\frac{a}{\varepsilon} \int_{t-\varepsilon/2}^{t+\varepsilon/2} \varphi_{k,+}(s) \dd s.
\end{equation}

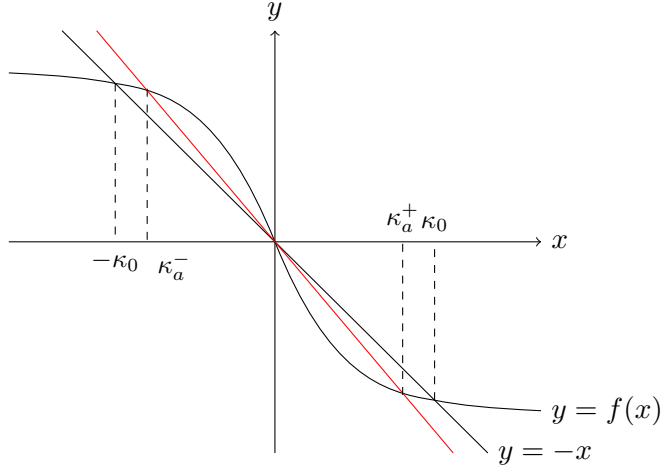
\begin{figure}[t]
    \centering
    \begin{tikzpicture}[scale=0.7]
    \draw[->] (-5,0) -- (5,0) node [anchor=west] {$x$};
    \draw[->] (0,-4) -- (0,4) node [anchor=south] {$y$};
    \draw (-4,4) -- (4,-4) node [anchor=west] {$y=-x$};
    \draw (-5,3.2) .. controls (-4,3.15) and (-3.5,3.1) .. (-3,3) .. controls (-2,2.85) and (-1,2.5) .. (0,0) .. controls (1,-2.5) and (2,-2.85) .. (3,-3) .. controls (3.5,-3.1) and (4,-3.15) .. (5,-3.2) node [anchor=west] {$y=f(x)$};
    \draw[red] (-3.35,4) -- (3.35,-4);
    \draw[dashed] (3,-3) -- (3,0) node [anchor=south] {\small{$\kappa_0$}};
    \draw[dashed] (-3,3) -- (-3,0) node [anchor=north] {\small{$-\kappa_0$}};
    \draw[dashed] (-2.4,2.9) -- (-2.4,0) node[anchor=north west] {\small{$\kappa_a^{-}$}};
    \draw[dashed] (2.4,-2.9) -- (2.4,0) node[anchor=south] {\small{$\kappa_a^+$}};
    \end{tikzpicture}
    \caption{Graph of the function $f\colon \R \to \R$ in a symmetric case.}
    \label{fig:4}
\end{figure}

On the other hand, since the restriction of the solutions $b^\varepsilon$ to $[-1-\varepsilon/2,0]$ belongs to $\B^\alpha$ and they are periodic with minimal period, then we deduce all of them have a period of the form $2+z(\varepsilon)$, where $|z(\varepsilon)|\leq \varepsilon$. Moreover, since the equation
\[
    b^\varepsilon(t) = \frac{1}{\varepsilon} \int_{t-1-\varepsilon/2}^{t-1+\varepsilon/2} f(b^\varepsilon(s)) \dd s
\]
is translation invariant, without loss of generality we assume that  $b^\varepsilon(0)=0$ and $\tau_\varepsilon\in(1-\varepsilon/2, 1+\varepsilon/2)$ is its associated first zero.

\begin{lem}
Let $I\subset (0,1)$ be a closed interval and $a\in(f'(0),-1)$ as before. Then there exists $k_0\in\N$ and $\varepsilon_0>0$, independent of $a$, such that for every $k\geq k_0$ with $k\equiv 1 \mod{4}$ and $\varepsilon\leq \varepsilon_0$ there holds
\begin{equation}
	b^\varepsilon(t) \geq \min\left\{ \kappa_a^+, \frac{|\kappa_a^-|}{\lambda_{a,k}} \right\} =: \mu_{a,k}, \quad t\in I.
\end{equation}
\end{lem}

\begin{proof}
Let $k_0$ and $\varepsilon_0$ large and small enough, respectively, such that $0<\min I -1/2k_0 < \max I + 1/2k_0<1-\varepsilon_0/2$. In particular, this implies that $\max I + 1/2k_0 < \tau_\varepsilon$ for every $\varepsilon\leq \varepsilon_0$. Now, take $T\in I$ and $k\geq k_0$ with $k \equiv 1 \mod{4}$ and set
\[
    \mu_{T,\varepsilon}^* = \sup \left\{ \mu\geq 0 \colon b^\varepsilon(t) \geq \mu\varphi_k^T(t), \quad t\in\left[ T- \frac{1}{2k},T+\frac{1}{2k} \right] \right\},
\]
where $\varphi_k^T(t) = \varphi_k(t+1/2-T)$. This value is well defined since $b^\varepsilon$ is bounded below by the barrier of eigenfunctions $\gamma_{\tau_\varepsilon}$ defined in Assumption \ref{as:Balpha}. Then, for $t\in\left[ T- \frac{1}{2k},T+\frac{1}{2k} \right]$ there holds 
\begin{align*}
    b^{\varepsilon}(t+2+z(\varepsilon)) &= \frac{1}{\varepsilon} \int_{t+z(\varepsilon)-\varepsilon/2}^{t+z(\varepsilon)+\varepsilon/2} f(b^\varepsilon(s+1)) \dd s \\
    &= \frac{1}{\varepsilon} \int_{t+ z(\varepsilon)-\varepsilon/2}^{t+ z(\varepsilon)+\varepsilon/2} f\left( \frac{1}{\varepsilon} \int_{s-\varepsilon/2}^{s+\varepsilon/2} f(b^\varepsilon(r)) \dd r \right) \dd s.
\end{align*}

We shall prove that there exists $\varepsilon_0>0$ such that for every $\varepsilon\leq \varepsilon_0$ we have $\mu_{T,\varepsilon}^* \geq \mu_{a,k}$. If we assume by contradiction that $\mu_{T,\varepsilon}^* < \min \{ \kappa_a^+, |\kappa_a^-|/\lambda_{a,k} \}$ for some $\varepsilon$ as small as we want, then owing to \eqref{eq:diag_condition2} we get $f(\mu_{T,\varepsilon}^*) \leq a \mu_{T,\varepsilon}^*$ and $f(-\lambda_a \mu_{T,\varepsilon}^*) \geq -a \lambda_a \mu_{T,\varepsilon}^*$, which jointly with \eqref{eq:subsolution} imply, for $\varepsilon$ small,
\begin{align*}
    \frac{1}{\varepsilon} \int_{s-\varepsilon/2}^{s+\varepsilon/2} f(b^\varepsilon(r)) \dd r &\leq \frac{1}{\varepsilon} \int_{s-\varepsilon/2}^{s+\varepsilon/2} f(\mu_{T,\varepsilon}^* \varphi_{k,+}^T(r)) \dd r \\
    &\leq \frac{\mu_{T,\varepsilon}^* a}{\varepsilon} \int_{s-\varepsilon/2}^{s+\varepsilon/2} \varphi_{k,+}^T(r) \dd r \\
    &\leq -\lambda_{a,k} \mu_{T,\varepsilon}^* \varphi_k^T(s).
\end{align*}
Consequently, applying \eqref{eq:lambda_ak} we obtain
\begin{align*}
    b^\varepsilon(t) &\geq \frac{1}{\varepsilon} \int_{t+z(\varepsilon)-\varepsilon/2}^{t+z(\varepsilon)+\varepsilon/2} f(-\lambda_{a,k} \mu_{T,\varepsilon}^* \varphi_k^T(s)) \dd s \\
    &\geq -\frac{a\lambda_{a,k} \mu_{T,\varepsilon}^*}{\varepsilon} \int_{t+z(\varepsilon)-\varepsilon/2}^{t+z(\varepsilon)+\varepsilon/2} \varphi_k^T(s) \dd s \\
    &= \lambda_{a,k}^2 \mu_{T,\varepsilon}^* \varphi_k^T(t+z(\varepsilon)).
\end{align*}

Now, by the definition of $\mu_{T,\varepsilon}^*$ there exists $t_0\in \left[ T- \frac{1}{2k},T+\frac{1}{2k} \right]$ such that $b^\varepsilon(t_0) = \mu_{T,\varepsilon}^* \varphi_k^T(t_0)$. Thus, using a Taylor expansion it follows
\begin{align*}
    b^\varepsilon(t_0) & \geq \lambda_{a,k}^2 \mu_{T,\varepsilon}^* \varphi_k^T (t_0 + z(\varepsilon)) \\
    & = \lambda_{a,k}^2 \mu_{T,\varepsilon}^* [ \varphi_k^T(t_0) + O(\varepsilon) ] \\
    & = \lambda_{a,k}^2 b^\varepsilon(t_0) + O(\varepsilon) \\
    &\geq \lambda_{a,k}^2 b^\varepsilon (t_0) - C\varepsilon,
\end{align*}
for a fixed constant $C>0$. Then,
\begin{equation}\label{eq:bound-b-epsilon}
    b^\varepsilon (t_0) \leq \frac{C \varepsilon}{\lambda_{a,k}^2 - 1}.
\end{equation}

Furthermore, as $b^\varepsilon\in \B^\alpha$ we know that $b^\varepsilon(t)\geq \gamma_{\tau_\varepsilon}(t)$ for all $t\in[0,\tau_\varepsilon]$, and minimizing the right-hand side of this inequality on the interval $J:=\left[ \min I - \frac{1}{2k}, \max I + \frac{1}{2k} \right]$ yields
\begin{align*}
    b^\varepsilon(t) &\geq \min\left\{ \gamma_{\tau_\varepsilon}\left( \min I - \frac{1}{2k} \right), \gamma_{\tau_\varepsilon}\left( \max I + \frac{1}{2k} \right) \right\} \\
    &= \alpha \min \left\{ \sin\left( \pi\left( \min I - \frac{1}{2k} \right)  \right) , \sin\left( \pi \left( \tau_\varepsilon -\max I -\frac{1}{2k} \right) \right) \right\},
\end{align*}
for every $t\in J$. Next, since $\tau_\varepsilon\in(1-\varepsilon/2,1+\varepsilon/2)$ the following bound holds
\[
    b^\varepsilon(t) \geq \alpha \min\left\{ \sin\left( \pi \left( \min I - \frac{1}{2k} \right) \right), \sin\left( \pi \left( 1 - \max I -\frac{1}{2k} - \frac{\varepsilon}{2} \right) \right) \right\}, \quad \forall t\in J.
\]
Observe this inequality shows that $b^\varepsilon(t)$ keeps positive in a uniform way with respect to $\varepsilon$, and this is a contradiction with \eqref{eq:bound-b-epsilon} for $\varepsilon$ small enough.

As a result, there exists $\varepsilon_0>0$ small enough such that for all $\varepsilon\leq \varepsilon_0$ there holds $\mu_{T,\varepsilon}^* \geq \mu_{a,k}$ and therefore
\[
    b^\varepsilon(t) \geq \min\left\{ \kappa_a^+, \frac{|\kappa_a^-|}{\lambda_{a,k}} \right\} \varphi_k^T(t), \quad t\in \left[T- \frac{1}{2k}, T+\frac{1}{2k} \right],
\]
from where, evaluating at $t=T$ and since $T$ is an arbitrary element of the interval $I$ the result follows.
\end{proof}

\begin{proof}[Proof of Theorem \ref{thm:limit-behavior}]
Using the hypothesis over the boundedness of the family $\{ b^\varepsilon \}_\varepsilon$ and the previous Lemma, we get
\[
     \kappa_0 \geq b^\varepsilon(t) \geq \min\left\{ \kappa_a^+, \frac{|\kappa_a^-|}{\lambda_{a,k}} \right\}, \quad t\in I,
\]
and thus
\[
    \max_{t\in I} |b^\varepsilon(t) - \kappa_0| \leq \left| \kappa_0 - \min\left\{ \kappa_a^+, \frac{|\kappa_a^-|}{\lambda_{a,k}} \right\} \right|,
\]
from where it is clear that, taking $\varepsilon\to 0^+$ and $a\to -1^{-}$,
\[
    \limsup_{\varepsilon\to 0^+} \max_{t\in I} |b^\varepsilon(t) - \kappa_0| = 0.
\]
We have proved therefore that $b^\varepsilon\to \kappa_0$ uniformly on $I$. In a similar fashion is possible to prove that $b^\varepsilon \to -\kappa_0$ uniformly on every compact subset of $(1,2)$ and then the first statement of the theorem follows.

The assertion concerned to the convergence in the $L^1_{loc}$-norm is proved in the following way. Observe that for every $n\in \N$ we have
\begin{align*}
    \int_{-2n}^{2n} |b^\varepsilon(t) - b^*(t)| \dd t &= \sum_{k=-n}^{n-1} \int_{2k}^{2(k+1)} |b^\varepsilon(t) - b^*(t)| \dd t \\
    &= \sum_{k=-n}^{n-1} \int_0^2 |b^\varepsilon(t+2k) - b^*(t+2k)| \dd t \\
    &= \sum_{k=-n}^{n-1} \int_0^2 |b^\varepsilon(t - kz(\varepsilon)) - b^*(t)| \dd t \\
    &= \sum_{k=-n}^{n-1} \int_{-kz(\varepsilon)}^{2-kz(\varepsilon)} |b^\varepsilon(t) - b^*(t + kz(\varepsilon))| \dd t \\
    &\leq \sum_{k=-n}^{n-1} \int_{-nz(\varepsilon)}^{2+nz(\varepsilon)} |b^\varepsilon(t) - b^*(t+kz(\varepsilon))| \dd t,
\end{align*}
and the result follows by a simple application of the dominated convergence theorem since $b^\varepsilon\to b^*$ and $b^*(\cdot + kz(\varepsilon))\to b^*$ pointwise almost everywhere as $\varepsilon\to 0$.
\end{proof}

\section{The Gurtin-MacCamy population model}\label{sec:gurtin}

Next we show how our results can shed some light on the limit behavior of the solutions to the Gurtin-MacCamy population model with a nonlinear birth function:
\begin{equation}\label{GM}
    \begin{cases}
        \ds (\partial_t + \partial_a)u(t,a) = -\mu u(t,a), & t>0,~ a>0 \\
        \ds u(t,0) = f\left( \int_0^\infty \gamma(a)u(t,a)\dd a \right), & t>0, \quad u(0,\cdot)=u_0\in L^1_+(\R_+),
    \end{cases}
\end{equation}
where $f$ is a unimodal function with $f(0)=f(+\infty)=0$ and for which the equation $f(x)=x$ has a unique positive solution that  we denote $\kappa$. Also, we assume that the kernel $\gamma(a)$ is normalized by
\[
    \int_0^\infty \gamma(a)e^{-\mu a}\dd a = 1.
\]

The method of characteristics yields the following solution to \eqref{GM}
\begin{equation}\label{eq:sol_GM}
    u(t,a) = \begin{cases}
        e^{-\mu t} u_0(a-t) & a\geq t,\\
        e^{-\mu a} b(t-a) & t\geq a,
    \end{cases}
\end{equation}
where $b\colon \R_+ \to \R_+$ is a solution for the equation
\begin{equation}\label{sol_b}
    b(t) = f\left( \int_t^{\infty} \gamma(a)e^{-\mu t} u_0(a-t)\dd a + \int_0^t \gamma(a)e^{-\mu a} b(t-a)\dd a \right), \quad t>0.
\end{equation}

In \cite{MMa}, the authors consider the model \eqref{GM} for the particular case of the Nicholson's (or Ricker's) nonlinearity, that is $f(x)=\alpha xe^{-x}$, and proved that for every $\alpha^\star>e^2$ there exists a kernel $\gamma(a)$ such that a Hopf bifurcation occurs at $\alpha=\alpha^*\in(e^2,\alpha^\star)$.

Due to \eqref{eq:sol_GM}, we can note that the asymptotic behavior of the solutions is determined by $b(t)$. Motivated by this, in \cite{HT} the authors consider the same model for a general class of unimodal functions satisfying the conditions stated above, and by means of a low-dimensional dynamical system approach they found out some conditions over the nonlinearity $f$ and the kernel $\gamma$ which implies that the equilibrium solution of the system, $\overline{u}(a)=\kappa e^{-\mu a}$, is globally attracting.

Let $(t_n)_{n\in \N}$ be a sequence of numbers such that $t_n\to +\infty$ and set $b_n(t) = b(t + t_n)$. We can pass to the limit by taking an appropriate subsequence which yields that the $\omega$-limit solutions for \eqref{sol_b} verify
\begin{equation}\label{eq:omegalimit}
    b(t) = f\left( \int_0^\infty \gamma(a)e^{-\mu a} b(t-a)\dd a \right), \quad t\in\R
\end{equation}
which is equivalent to the equation
\[
    B(t) = \int_0^\infty \gamma(a)e^{-\mu a} f(B(t-a))\dd a
\]
by taking $B(t)= \int_0^\infty \gamma(a)e^{-\mu a} b(t-a)\dd a$ and $b(t)=f(B(t))$. In addition, making the following choices
\begin{align}
    \gamma(a)e^{\mu a} &= \begin{cases}
        \frac{1}{\varepsilon},  & a\in[1-\varepsilon/2,1+\varepsilon/2], \\
        0 & \text{otherwise},
    \end{cases} \label{eq:delay_kernel}\\
    \zeta(t) &= B(t)+\kappa, \notag \\
    F(t) &= f(t+\kappa) - \kappa, \notag
\end{align}
then the last equation becomes
\[
    \zeta(t) = \frac{1}{\varepsilon}\int_{1-\varepsilon/2}^{1+\varepsilon/2} F(\zeta(t-a))\dd a,
\]
and the function $F$ satisfies the conditions needed in Theorem \ref{thm:existence-solutions} if we impose certain conditions over $f$. More precisely:

\begin{thm}\label{main_th2}
Let $f$ be a unimodal $C^1$-function with $f(0)=f(+\infty)=0$, $f(x)=x$ has a unique positive solution $\kappa$ and $f(x)>x$ for every $x\in(0,\kappa)$. In addition, let suppose $f'(\kappa)<-2$ and $f(\|f\|)> x_*$ where $x_*$ is the unique value less that $\kappa$ such that $f(x_*)=\kappa$. Then, the Gurtin-MacCamy population model \eqref{GM} with the delayed kernel \eqref{eq:delay_kernel} admits an asymptotically periodic solution. (Recall, $\|\cdot\|$ denotes the sup norm)
\end{thm}
\begin{proof}
As the asymptotic behavior of the solutions to \eqref{GM} relies on the solutions to \eqref{eq:omegalimit} with $f(\|f\|)\leq b(t)\leq \|f\|$ (see \cite[Theorem 1]{HT} and \cite[Proposition 5.12]{SmithThieme}), if we define
\[
    F(t) = \begin{cases}
        f(t+\kappa)-\kappa, & t\geq x_\star-\kappa,\\
        f(x_*) , & t\leq x_\star-\kappa,
    \end{cases}
\]
for any fixed $x_\star\in(x_*, f(\|f\|))$, then $F$ verifies the hypothesis of Theorem \ref{thm:existence-solutions} and the set of solutions remains unchanged. Thus, the application of this result yields the existence of a periodic orbit for \eqref{eq:omegalimit} with the kernel given by \eqref{eq:delay_kernel}.
\end{proof}

Thus, Theorem \ref{main_th2} shows that there are periodic orbits for the Gurtin-MacCamy model. Although this does not mean that all the solutions for this case have a periodic long-term behavior, since these questions rely on the stability of such periodic solution, it show us the quite variety and complexity of the dynamics this model encompasses. Actually, is worth mentioning that for $|f'(\kappa)|\gg 1$ we suspect the model will show a chaotic behavior.

\section{Numerical simulations}\label{sec:simulations}

In this section we provide numerical simulations done for explicit functions $f$. This lets us visualize the results stated before in a mathematical way, and also observe new phenomena which can be the subject of further studies around this type of problems.

All the numerical simulations were carried out by using the Simpson rule to approximate the integral. Also, we consider constant initial data, unlike the functions considered in the proof of the existence of periodic solutions, for simplicity.

\begin{figure}[t!]
    \centering
    \begin{subfigure}{0.4\textwidth}
        \centering
        \includegraphics[scale=0.45]{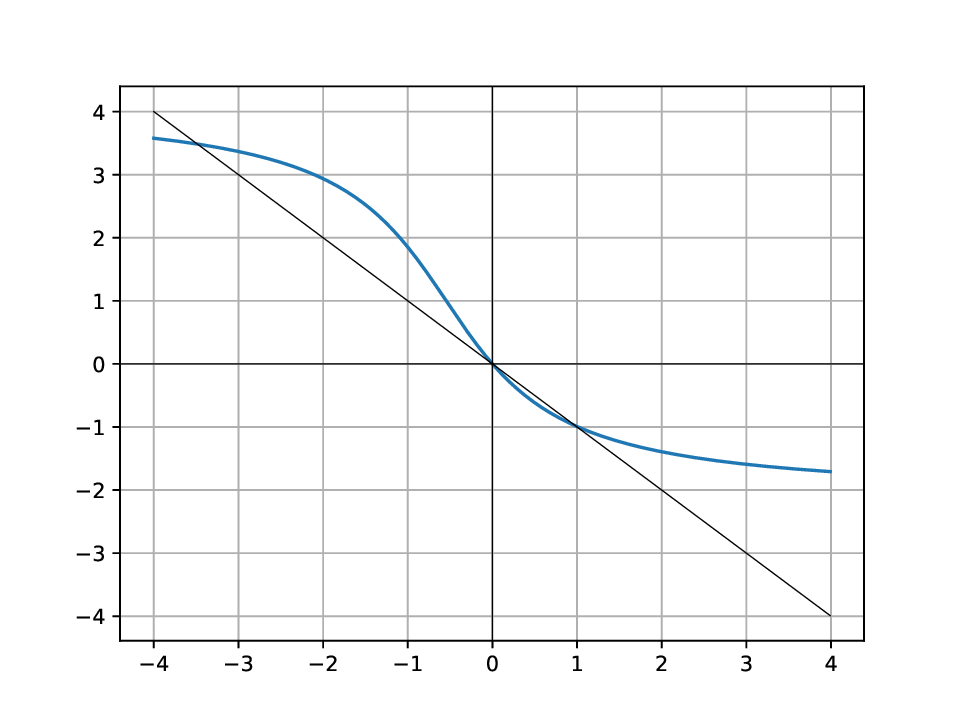}
        \caption{Non-symmetric but monotone case}
        \label{Fig:f_1}
    \end{subfigure}
    ~
    \begin{subfigure}{0.4\textwidth}
        \centering
        \includegraphics[scale=0.45]{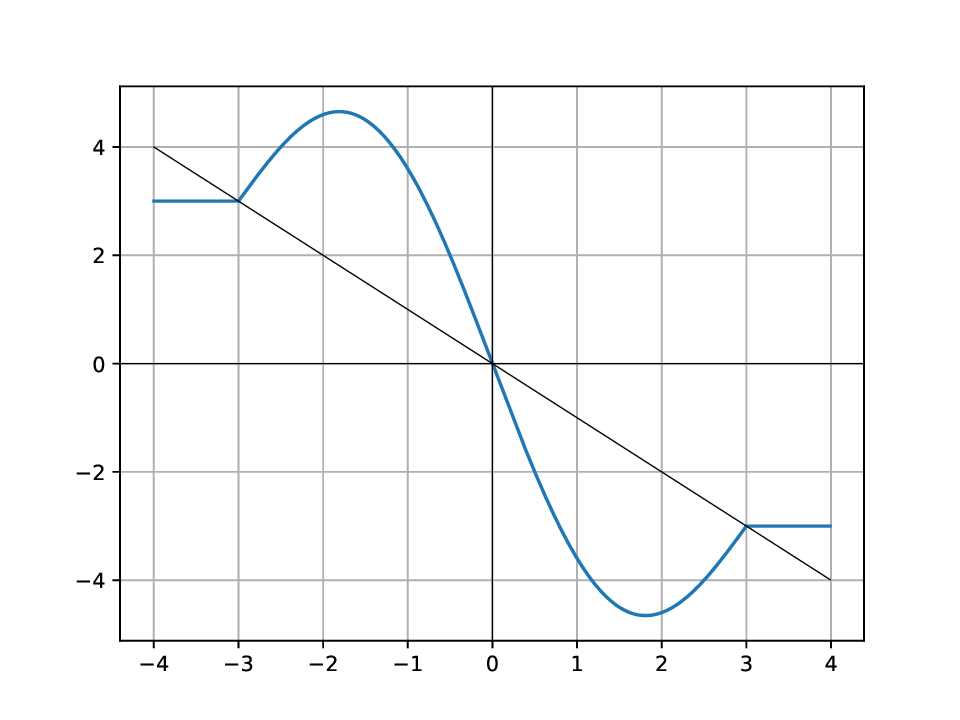}
        \caption{Symmetric but non-monotone case}
        \label{Fig:f_2}
    \end{subfigure}\\
    \begin{subfigure}{0.4\textwidth}
        \centering
        \includegraphics[scale=0.45]{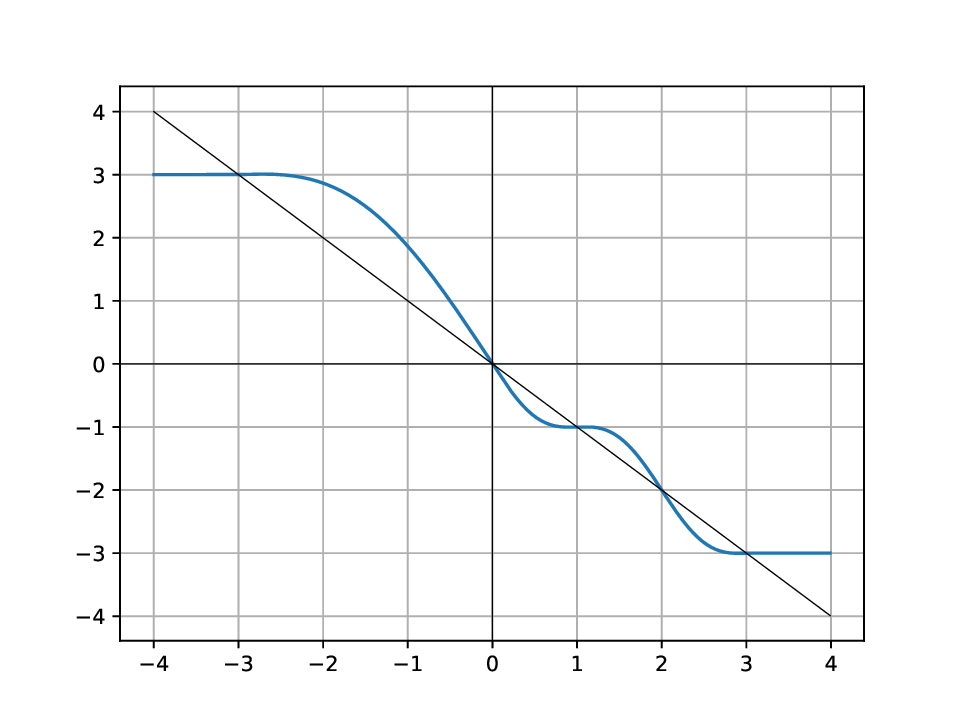}
        \caption{Case with several intersections with the diagonal}
        \label{Fig:f_3}
    \end{subfigure}
    \caption{Different cases for the nonlinearity $f\colon \R \to \R$ under consideration.}
    \label{fig:nonlinearities}
    
\end{figure}

\subsection{A non-symmetric but monotone nonlinearity}

Let $f(x) = -2 \tan^{-1}(x + \tan(1/2)) + 1$, see Figure \ref{Fig:f_1}. In this case we are dealing with the case which motivated this work, dropping out the symmetry condition used originally in \cite{CDM}, but keeping the monotone one. The simulations ran show the appearance of the periodic orbit, and, in accordance with the study of the difference equation given in \cite{sharkovsky2}, its extreme values are determined by the points of period two of the nonlinearity, see Figure \ref{fig:Simulations_1}. Also, as $\varepsilon$ decreases we can see the convergence to the square-wave solution.

\begin{figure}[h!]
    \centering
    \begin{subfigure}{0.4\textwidth}
    \centering
    \includegraphics[scale=0.45]{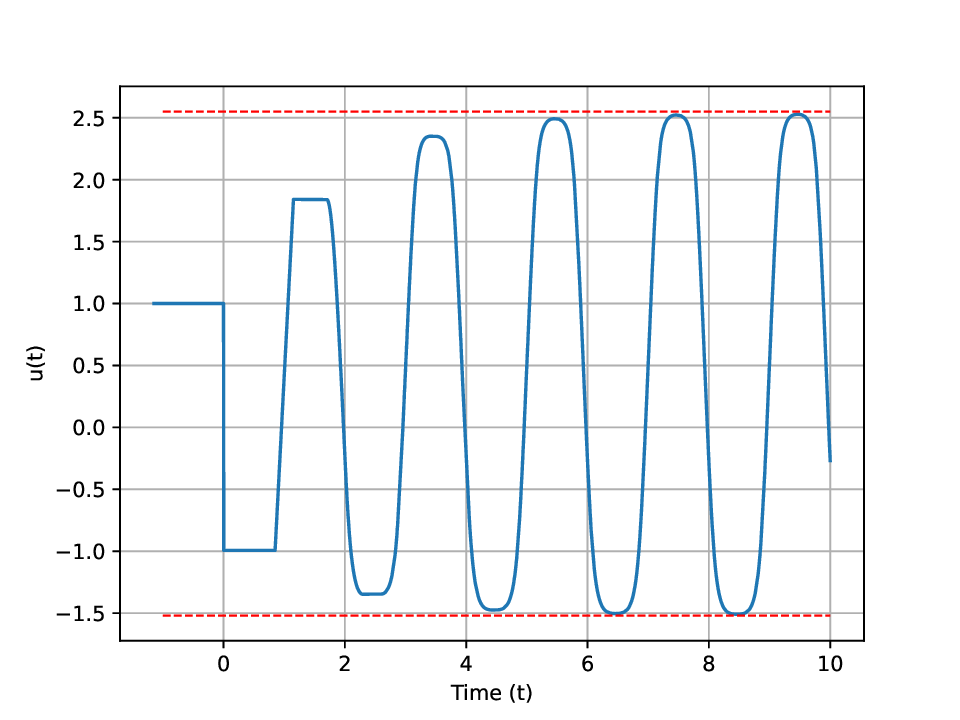}
    \caption{$\varepsilon=0.3$}
    \end{subfigure}
    ~
    \begin{subfigure}{0.4\textwidth}
    \includegraphics[scale=0.45]{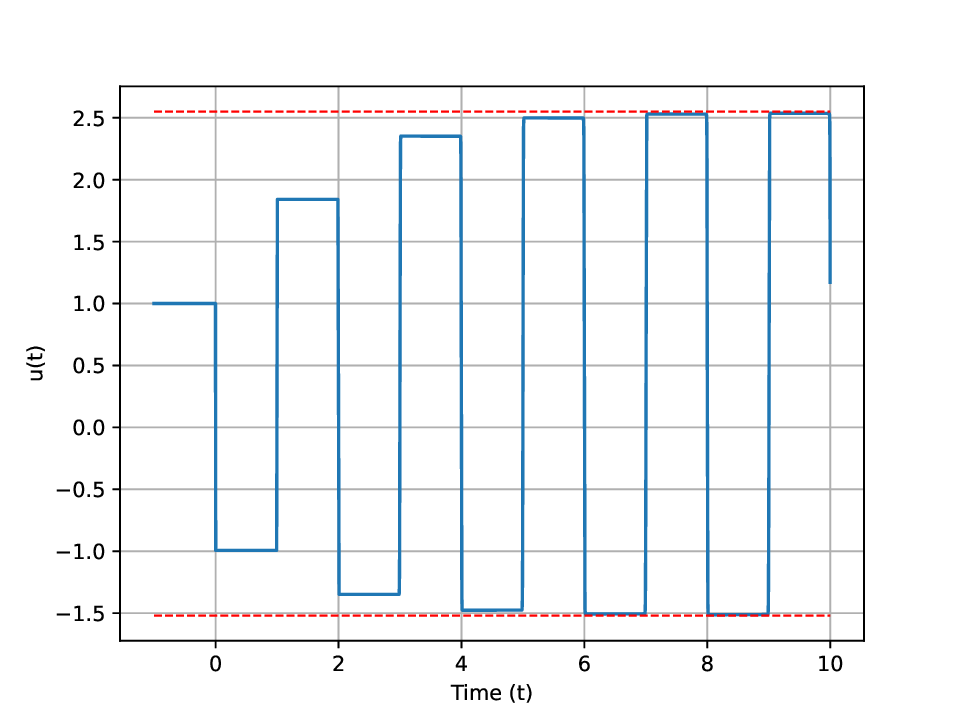}
    \caption{$\varepsilon=0.01$}
    \end{subfigure}
    \caption{Simulations (blue) for the case depicted in Figure \ref{Fig:f_1}, taking the constant function $u_0(t)=1$ as initial data. The solutions to the equation $f^2(x)=x$ are plotted in red.}
    \label{fig:Simulations_1}
\end{figure}

\subsection{A symmetric but non-monotone case}

Let
\[
    f(x) = \begin{cases}
        3, & x\leq -3, \\
        -x - 3 \sin\left( \frac{\pi x}{3} \right), & -3 \leq x \leq 3, \\
        3, & x\geq 3,
    \end{cases}
\]
see Figure \ref{Fig:f_2}. In this case, the function $f$ is non monotone which implies that $\max_{|x|\leq 3}|f(x)|> 3$. This particular example shows that, even though the square-wave solution is still determined by the period-two points of the function $f$, the convergence is not as regular as the one for the previous case. As we can see in Figure \ref{fig:Simulations_2}, a sort of Gibbs phenomenon is presented in the convergence towards the square-wave solution. This behavior was already observed by Mallet-Paret and Nussbaum for the singularly perturbed version of the difference equation, and they notice that it appears also in a non-monotone case. For more results and comments in this direction see \cite{nussbaum_mallet}.

\begin{figure}[h!]
    \centering
    \begin{subfigure}{0.4\textwidth}
    \centering
    \includegraphics[scale=0.45]{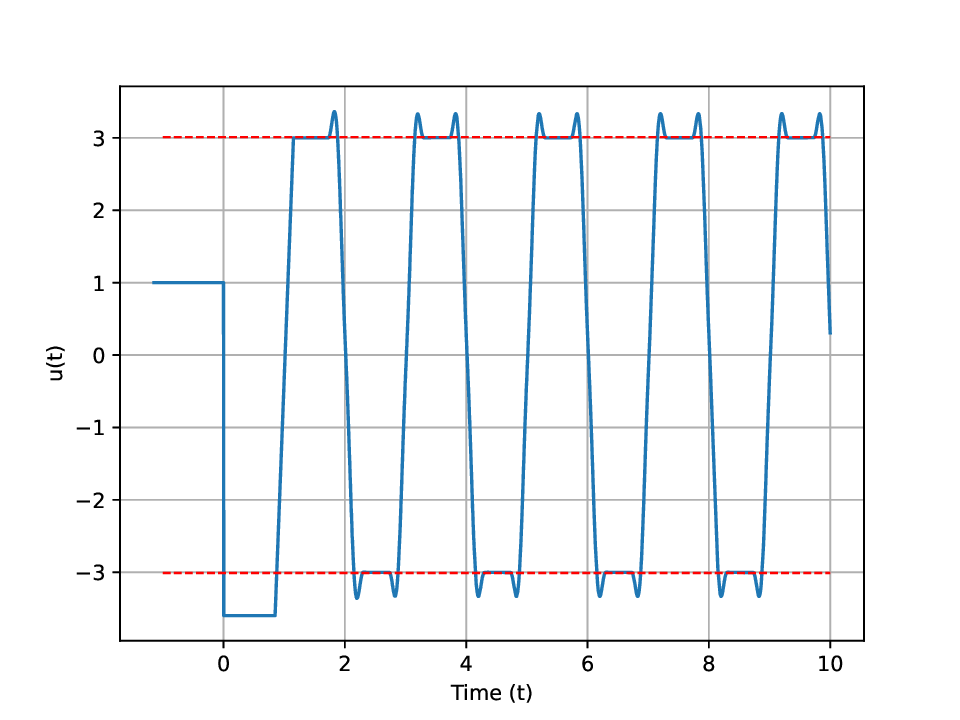}
    \caption{$\varepsilon=0.3$}
    \end{subfigure}
    ~
    \begin{subfigure}{0.4\textwidth}
    \centering
    \includegraphics[scale=0.45]{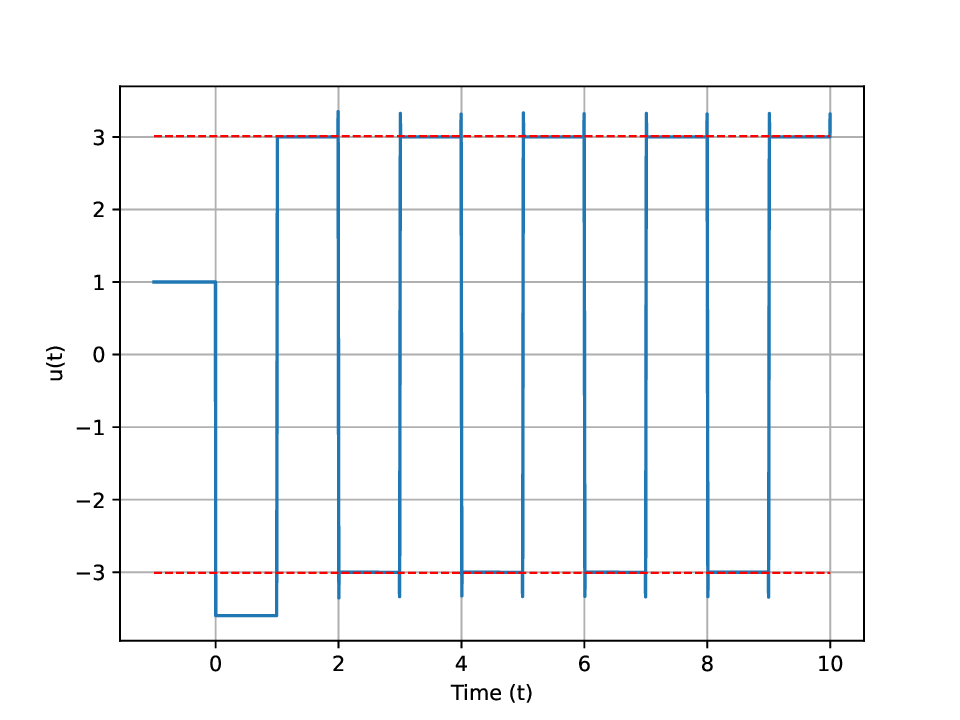}
    \caption{$\varepsilon=0.01$}
    \end{subfigure}
    \caption{Simulations (blue) for the case depicted in Figure \ref{Fig:f_2}, taking the constant function $u_0(t)=1$ as initial data. The solutions to the equation $f^2(x)=x$ are plotted in red.}
    \label{fig:Simulations_2}
\end{figure}

\subsection{A case with several intersections with the diagonal}

Let
\[
    f(x) = \begin{cases}
        3, & x\leq -3, \\
        -x- \sin\left( \frac{\pi x}{3} \right), & -3\leq x \leq 0, \\
        -x - \frac{1}{3} \sin\left( \pi x \right), & 0 \leq x \leq 3, \\
        3 & x\geq 3,
    \end{cases}
\]
see Figure \ref{Fig:f_3}. In this case we consider a non-symmetric function $f$ with several intersections with the diagonal. An a priori sight of the problem may led us to think these intersections will affect the behavior of the periodic orbit which arises. Nevertheless, again, this behavior is mainly determined by the period-two points this functions possesses, see Figure \ref{fig:simulations_3}. In this case, we recover a behavior similar to the one observed in the first example, and this is supported by the observation that this function $f$ has exactly two nontrivial solutions to the equation $f^2(x)=x$, see Figure \ref{fig:f_3-2}.

\begin{figure}[h!]
    \centering
    \begin{subfigure}{0.4\textwidth}
    \centering
    \includegraphics[scale=0.45]{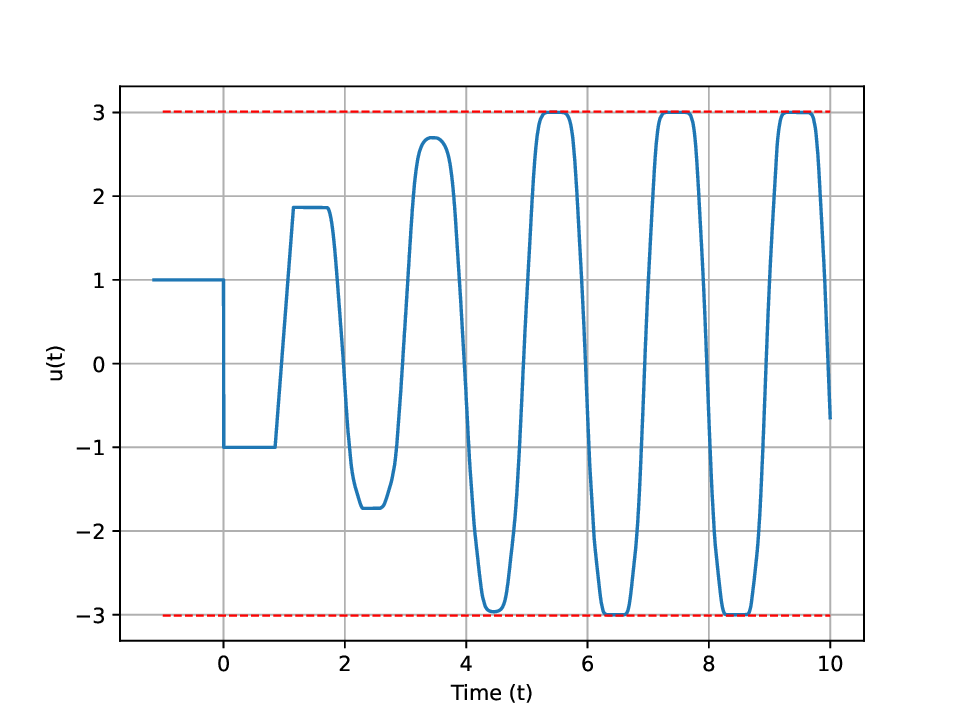}
    \caption{$\varepsilon=0.3$}
    \end{subfigure}
    ~
    \begin{subfigure}{0.4\textwidth}
    \centering
    \includegraphics[scale=0.45]{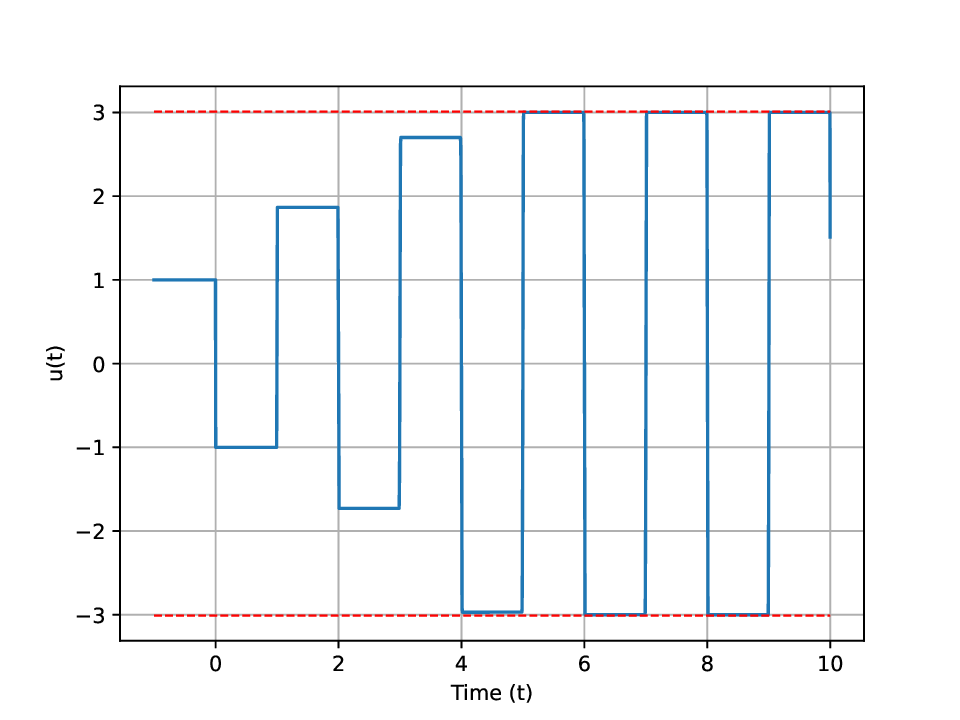}
    \caption{$\varepsilon=0.01$}
    \end{subfigure}
    \caption{Simulations (blue) for the case depicted in Figure \ref{Fig:f_3}, taking the constant function $u_0(t)=1$ as initial data. The solutions to the equation $f^2(x)=x$ are plotted in red.}
    \label{fig:simulations_3}
\end{figure}

\begin{figure}[h!]
    \centering
    \includegraphics[width=0.5\linewidth]{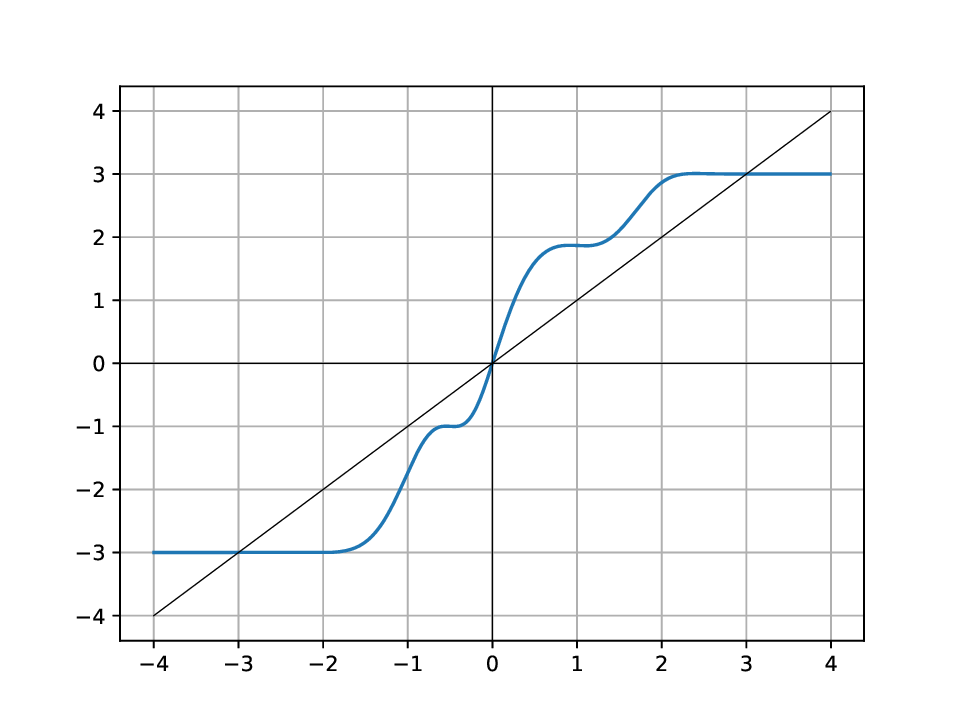}
    \caption{Graph of the function $f^2(x)$, with $f(x)$ as in Figure \ref{Fig:f_3}, compared with the diagonal $y=x$.}
    \label{fig:f_3-2}
\end{figure}

\bibliography{biblio.bib}
\bibliographystyle{abbrv}

\end{document}